\documentclass[12pt]{amsart}
\usepackage{amsmath,amssymb,mathrsfs,mathtools,lipsum,enumitem,titlesec}
\usepackage{secdot,  booktabs}
\usepackage{geometry}
\usepackage{amsthm,}
\usepackage[utf8]{inputenc}
\usepackage[english]{babel}
\usepackage{hyperref}

\hypersetup{
	colorlinks=true,
	linkcolor=blue,
	citecolor=blue,
	filecolor=blue,
	urlcolor=blue,
}
\usepackage{cleveref}
\usepackage{indentfirst}

\usepackage{thmtools}
\newtheorem{theorem}{Theorem}[section]
\newtheorem{definition}[theorem]{Definition}
\newtheorem{proposition}[theorem]{Proposition}
\newtheorem{remark}[theorem]{Remark}
\newtheorem{lemma}[theorem]{Lemma}
\newtheorem{example}[theorem]{Example}
\newtheorem{corollary}[theorem]{Corollary}

\setlist[enumerate]{topsep=8pt, itemsep=-2pt, partopsep=1ex, parsep= 1ex}
\setenumerate[1]{label={$(\arabic*)$}}
\usepackage{color,soul}
\usepackage{lineno}
\usepackage[symbol]{footmisc}
\usepackage{ragged2e}

\modulolinenumbers[1]

\begin{document}

\title[On $k-$WUR and its generalizations]{On $k-$WUR and its generalizations}

	\author{P. Gayathri}
\address{Department of Mathematics, National Institute of Technology Tiruchirappalli, Tiruchirappalli, Tamil Nadu - 620 015, India}
\email{pgayathriraj1996@gmail.com}
\thanks{}

\author{V. Thota}
\address{Department of Mathematics, National Institute of Technology Tiruchirappalli, Tiruchirappalli, Tamil Nadu - 620 015, India}
\email[Corresponding author]{vamsinadh@nitt.edu}
\thanks{}

\keywords{$k-$weakly uniformly rotund; $k-$weakly locally uniformly rotund;  $k-$weakly strongly Chebyshev; $k-$weakly uniformly strongly Chebyshev; property $k-$$w$UC; product space}
\dedicatory{}

\begin{abstract}
We introduce two notions called $k-$weakly uniform rotundity ($k-$WUR) and $k-$weakly locally uniform rotundity ($k-$WLUR) in real Banach spaces. These are natural generalizations of the well-known concepts $k-$UR and WUR. By introducing two best approximation notions namely $k-$weakly strong Chebyshevity and $k-$weakly uniform strong Chebyshevity, we generalize some of the existing results to $k-$WUR and $k-$WLUR spaces. In particular, we present characterizations of $k-$WUR spaces in terms of $k-$weakly uniformly strong Chebyshevness. Also, the inheritance of the notions $k-$WUR and $k-$WLUR by quotient spaces are discussed. Further, we provide a necessary and sufficient condition for an infinite $\ell_p-$product space to be $k-$WUR (respectively, $k-$WLUR). As a consequence, we observe that the notions WUR and $k-$WUR coincide for an infinite $\ell_p-$product of a Banach space.
\end{abstract}

\maketitle

\section{Introduction}

Let $X$ be a real Banach space and $X^*$ be its dual. The closed unit ball and the unit sphere of $X$ are denoted by $B_X$ and $S_X$ respectively. Various rotundity notions play significant role in the geometry of Banach spaces, best approximation theory and  fixed point theory.  As a more robust definition of rotundity, Clarkson \cite{Clar1936} proposed the concept of uniform rotundity in 1936. Then, as a directionalization of uniform rotundity in 1939, \v{S}mulian introduced weakly uniform rotundity \cite{Smul1939}. The corresponding local versions called, locally uniform rotundity and weakly locally uniform rotundity were later introduced by Lovaglia in 1955 \cite{Lova1955}. Additionally, generalizations of these local versions known as midpoint locally uniform rotundity and weakly midpoint locally uniform rotundity are introduced and studied in the literature. We refer to \cite{AlDi1985,DeGZ1993,DuSh2018,GaTh2022,Megg1998,Smit1978a,Smit1986,ZhLZ2015} for further study on these notions.

In addition to these notions, Sullivan established the notion of $k-$uniform rotundity \cite{Sull1979} in 1979 as a generalization of uniform rotundity using the $k-$dimensional volume that is presented as follows. For any $k \in \mathbb{Z}^{+},$ the $k-$dimensional determinant $D_k[x_1, x_2, \dots, x_{k+1};f_1, f_2, \dots, f_k]$ generated by $(k+1)-$vectors $x_1,x_2,\dots,x_{k+1}$ in $X$ and $k-$functionals $f_1,f_2,\dots,f_k$ in $X^*$ is defined as

\begin{align*}
	D_k[x_1, x_2, \dots, x_{k+1};f_1, f_2, \dots, f_k]= &
	\begin{vmatrix}
		1 & 1 & \dots & 1 \\
		f_1(x_1) & f_1(x_2) & \dots & f_1(x_{k+1})\\
		\vdots & \vdots & \ddots & \vdots \\
		f_k(x_1) & f_k(x_2) & \dots & f_k(x_{k+1})\\
	\end{vmatrix}.\\
\end{align*}

\justifying{ The $k-$dimensional volume $V[x_1, x_2, \dots, x_{k+1}]$  enclosed by $(k+1)-$vectors $x_1, x_2, \dots, x_{k+1}$ in $X$ is defined as
$V[x_1, x_2, \dots, x_{k+1}]=  \sup \{D_k[x_1, x_2, \dots, x_{k+1};g_1, g_2, \dots, g_k]: g_1,g_2,\dots,g_k \in S_{X^*}\}.$
We also write $ D_k[x_1, x_2, \dots, x_{k+1};f_1, f_2, \dots, f_k]$ as $D_k[(x_i)_{i=1}^{k+1}; (f_j)_{j=1}^{k}]$ and $V[x_1, x_2, \dots, x_{k+1}]$ as $V[(x_i)_{i=1}^{k+1}].$	}

\begin{definition}\textnormal{\cite{Sull1979}} Let $k \in \mathbb{Z}^{+}.$ A space $X$ is said to be $k-$uniformly rotund (in short, $k-$UR), if for every $\epsilon>0$,
	\[
	\delta^{k}_{X}(\epsilon) \coloneqq  \inf\left\{1-\dfrac{1}{k+1}\left \Vert \sum\limits_{i=1}^{k+1}x_i \right \Vert: 	x_1, x_2, \dots,x_{k+1}\in S_X,
	V[(x_i)_{i=1}^{k+1}]\geq \epsilon
	\right\}>0.
	\]
\end{definition}
It is significant to note that the notion $k-$UR reinforces Singer's concept of $k-$rotundity \cite{Sing1960}, which is stated in the following manner. The space $X$ is said to be $k-$rotund, if for any $x_1, x_2, \dots, x_{k+1} \in S_X$ with  $V[(x_i)_{i=1}^{k+1}]>0$, it follows that $\frac{1}{k+1}\|\sum_{i=1}^{k+1}x_i\|=1.$
Similar to this, $k-$version of some other rotundity properties are also introduced and studied in the literature as $k-$locally uniform rotundity (in short, $k-$LUR) \cite{Sull1979}, $k-$midpoint locally uniform rotundity (in short, $k-$MLUR) \cite{He1997}, $k-$weakly midpoint locally uniform rotundity (in short, $k-$WMLUR) \cite{XiLi2004} and $k-$strong rotundity \cite{VeRS2021}. We refer to \cite{GaTh2023,GeSu1981,KaVe2018,LiYu1985,LiZZ2018,NaWa1988,SmTu1990,Veen2021,YuWa1990} for further study on these notions.

It would seem relevant to investigate the question of whether it is feasible to define the concepts of WUR and WLUR in a corresponding $k-$version. In response to this question, we introduce the notions $k-$WUR and $k-$WLUR (see, \Cref{def k-wur}) in this article. These generalizations are natural in Sullivan's sense.

It is well established in the literature that the geometry of Banach spaces and the best approximation theory in Banach spaces are closely related. Several authors investigated various rotundity properties  in terms of notions from best approximation theory.  To see this, we need the following notations and notions.

Let $k \in \mathbb{Z}^{+}.$ For any non-empty bounded subset $C$ of $X,$ the $k-$dimensional diameter $diam_k(C)$ is defined as $diam_k(C)=\sup\{V[(x_{i})_{i=1}^{k+1}]: x_1, x_2, \dots, x_{k+1} \in C\}.$ For any non-empty  subset $A$ of $X,$ $x\in X$ and $\delta>0,$ we define $P_A(x)=\{y \in A: \|x-y\|=d(x,A)\}$ and $P_A(x, \delta)=\{y \in A: \|x-y\| \leq d(x, A)+ \delta\},$ where $d(x,A)= \inf\{\|x-y\|: y \in A\}.$

The set $A$ is said to be proximinal at $x,$ if $P_A(x)$ is non-empty.  We say that $A$ is $k-$Chebyshev \cite{Sing1960} at $x,$ if A is proximinal at $x$ and  $diam_k(P_A(x))=0.$ We say that $A$ is $k-$strongly Chebyshev (in short, $k-$SCh) \cite{VeRS2021} at $x,$ if $A$ is proximinal at $x$ and  for every $\epsilon>0$ there exists $\delta>0$ such that $diam_k(P_A(x, \delta))< \epsilon.$ Let $B$ be any non-empty subset of $X.$  The set $A$ is said to be proximinal on $B,$ if $A$ is proximinal at every $x \in B.$ Similarly, we define for other notions such as $k-$Chebyshev and $k-$strongly Chebyshev. We say that $A$ is $k-$uniformly strongly Chebyshev (in short, $k-$USCh) \cite{KaVe2018} on $B,$ if $A$ is proximinal on $B$ and  for every $\epsilon>0$ there exists $\delta>0$ such that $diam_k(P_A(x, \delta))< \epsilon$ for all $x \in B.$

 According to Singer  \cite{Sing1960}, a space is $k-$rotund if every proximinal convex subset is $k-$Chebyshev, and the converse also holds. Recently, a characterization of $k-$UR spaces in terms of $k-$USCh obtained by Kar and Veeramani \cite{KaVe2018}. Also, a characterization of $k-$MLUR spaces in terms of $k-$SCh is obtained in \cite{LiZZ2018}. Further, Veena Sangeetha et al. \cite{VeRS2021} proved that the space is $k-$strongly rotund iff every closed convex subset is $k-$SCh.

In light of the aforementioned results, it is reasonable to presume that the new geometric notions $k-$WUR and $k-$WLUR may be characterized or analyzed in terms of appropriate notions from best approximation theory. To achieve this, two best approximation notions namely $k-$weakly strong Chebyshevness and  $k-$weakly uniformly strong Chebyshevness (see, \Cref{def k-SCh}) are defined in this article.

It has been extensively explored and is of great interest how quotient spaces may inherit rotundity notions and how geometric notions can be stable under $\ell_p-$product. Geremia and Sullivan \cite{GeSu1981} proved that for any $1<p<\infty,$ $(\oplus_p X_i)_{i\in\mathbb{N}}$ is $2-$UR iff all but except one of the $X_i$ are UR with a common modulus of convexity and the remaining space is $2-$UR. Later, in \cite{YuWa1990}, the authors generalized this result for any $k\in \mathbb{Z^+}.$ In \cite{SmTu1990}, Smith and Turett proved that an $\ell_p-$product space can not be $k-$UR whenever none of the underlying space is UR. Recently, all the preceding results were obtained for $k-$rotund spaces in an analogous way \cite{Veen2021,VeRS2021}. This article seeks to examine the stability of the notions $k-$WUR and $k-$WLUR in this approach.

The paper is organized as follows. In Section 2, we present some properties of $k-$dimensional determinants, which are essential to prove our results. We introduce the notions $k-$weakly strong Chebyshevness (in short, $k-$$w$SCh), $k-$weakly uniformly strong Chebyshevness (in short, $k-$$w$USCh), property $k-$$w$UC and obtain some relationships among them. These notions will be used to provide some necessary and sufficient conditions for a space to be $k-$WUR (respectively,  $k-$WLUR,  $k-$WMLUR).

In Section 3, we introduce and study the notions $k-$WUR and $k-$WLUR. Some of the sequential characterizations of $k-$WUR and $k-$WLUR presented in this section are necessary to prove our main results.
Characterizations of $k-$WUR spaces in terms of $k-$$w$USCh of the closed unit ball as well as subspaces are obtained. Further, the concepts $k-$$w$SCh and property $k-$$w$UC are explored in $k-$WLUR and $k-$WMLUR spaces. We provide counter examples to demonstrate that the converses of some implications are not necessarily true.

We investigate the stability of the notions $k-$WUR, $k-$WLUR and $k-$WMLUR in Section 4. First, we obtain two results that correlate $k-$WUR and $k-$WLUR properties of a space with the associated quotient spaces.
We provide necessary and sufficient conditions for finite and infinite $\ell_p-$product space to be $k-$WUR (respectively, $k-$WLUR, $k-$WMLUR). As a consequence, we observe that the notions WUR (respectively, WLUR, WMLUR) and $k-$WUR (respectively, $k-$WLUR, $k-$WMLUR) coincide for an infinite $\ell_p-$product of a Banach space.
	
	\section{Preliminaries}
	This section begins with some $k-$dimensional determinant properties that will be utilized throughout the article. We assume all subspaces to be closed.
\begin{remark} \label{volume property}
		Let $x_1, x_2, \dots, x_{k+1} \in X$ and $f_1, f_2, \dots, f_k \in X^*.$ Then,
		\begin{enumerate}\label{remark vol}
			\item $D_k[(x_i+w)_{i=1}^{k+1};(f_j)_{j=1}^{k}]= D_k[(x_i)_{i=1}^{k+1};(f_j)_{j=1}^{k}],$ for any $w \in X;$
			\item $D_k[(cx_i)_{i=1}^{k+1};(f_j)_{j=1}^{k}]$ = $c^kD_k[(x_i)_{i=1}^{k+1};(f_j)_{j=1}^{k}],$ for any $c \in \mathbb{R}$;
			\item $D_k[(x_i)_{i=1}^{k+1};(g_j)_{j=1}^{k}]>0$ for some $g_1, g_2, \dots, g_k \in S_{X^*}$  $\Leftrightarrow$ the set $\{x_i-x_{k+1}: 1 \leq i \leq k\}$ is linearly independent;
			\item $D_k[(x_i+Y)_{i=1}^{k+1}; (h_j)_{j=1}^{k}] = D_k[(x_i+y_i)_{i=1}^{k+1}; (h_j)_{j=1}^{k}],$ where $Y$ is a subspace of $X,$  $x_i+Y \in X/Y,$ $y_i \in Y$ for all $1 \leq i \leq k+1$ and $h_j \in Y^{\perp} \cong (X/Y)^*$ for all $1 \leq j \leq k.$
		\end{enumerate}
	\end{remark}
In the following result, we observe certain continuity properties of  $k-$dimensional determinants. The proof of \cite[Lemmas 2.9 and 2.10]{KaVe2018} may  also be used to prove the following lemma. However, we concisely provide a direct proof here.
\begin{lemma}\label{lem samir}
	Let $(x_n^{(1)}),(x_n^{(2)}), \dots, (x_n^{(k+1)})$ be   $(k+1)-$bounded sequences in $X$. Then for any $f_1, f_2, \dots, f_{k}  \in S_{X^*}$ the following statements hold.
	\begin{enumerate}
		\item If $(c_n^{(1)}),(c_n^{(2)}), \dots, (c_n^{(k+1)})$ be  $(k+1)-$sequences in $\mathbb{R}$ such that $c_n^{(i)} \to c_i$ for some $c_i \in \mathbb{R},$ for all $1 \leq i \leq k+1,$ then  $\vert D_k[(c_n^{(i)}x_n^{(i)})_{i=1}^{k+1};(f_j)_{j=1}^{k}]- D_k[(c_ix_n^{(i)})_{i=1}^{k+1};(f_j)_{j=1}^{k} ]\vert \to 0.$ 	
		\item If $(y_n^{(1)}),(y_n^{(2)}), \dots, (y_n^{(k+1)})$ be $(k+1)-$sequences in $X$ such that $y_n^{(i)} \xrightarrow{w} y_i$ for some $y_i \in X,$  for all $1 \leq i\leq k+1,$ then
		$\vert D_k[(y_n^{(i)}+x_n^{(i)})_{i=1}^{k+1};(f_j)_{j=1}^{k}]-  D_k[(y_i+x_n^{(i)})_{i=1}^{k+1};(f_j)_{j=1}^{k}]\vert \to 0.$
	\end{enumerate}
\end{lemma}
\begin{proof}
	$(1)$: Let $(c_n^{(1)}), (c_n^{(2)}), \dots, (c_n^{(k+1)})$ be $(k+1)-$sequences in $\mathbb{R}$ such that $c_n^{(i)} \to c_i$ for some $c_i \in \mathbb{R},$ for all $1 \leq i \leq k+1.$ Let $a_n= (D_k[(c_n^{(i)}x_n^{(i)})_{i=1}^{k+1};(f_j)_{j=1}^{k}]- D_k[(c_ix_n^{(i)})_{i=1}^{k+1};(f_j)_{j=1}^{k}])$ for every  $n \in \mathbb{N}.$
\[
		\begin{aligned}
			|a_n|=& \vert D_k[(c_n^{(i)}x_n^{(i)})_{i=1}^{k+1};(f_j)_{j=1}^{k}]+ \sum_{i=1}^{k}D_k[(c_n^{(j)}x_n^{(j)})_{j=1}^{i}, (c_jx_n^{(j)})_{j=i+1}^{k+1}; (f_j)_{j=1}^{k}]- \\
			&\hspace{1.5cm} \sum_{i=1}^{k}D_k[(c_n^{(j)}x_n^{(j)})_{j=1}^{i}, (c_jx_n^{(j)})_{j=i+1}^{k+1}; (f_j)_{j=1}^{k}]-	
			D_k[(c_ix_n^{(i)})_{i=1}^{k+1};(f_j)_{j=1}^{k} ]\vert 	\\
			\leq &\sum_{i=1}^{k+1}|D_k[(c_n^{(j)}x_n^{(j)})_{j=1}^{i}, (c_jx_n^{(j)})_{j=i+1}^{k+1}; (f_j)_{j=1}^{k}]-D_k[(c_n^{(j)}x_n^{(j)})_{j=1}^{i-1}, (c_jx_n^{(j)})_{j=i}^{k+1}; (f_j)_{j=1}^{k}]|.\\
		\end{aligned}
\]
Using the properties of determinant, for any $n \in \mathbb{N}$ we have $ |a_n| \leq \sum_{i=1}^{k+1}|b_n^{(i)}|,$ where
\[
	\begin{aligned}
		 b_n^{(i)} &= \begin{vmatrix}
			1 & \dots & 1 & 0 & 1 & \dots & 1 \\
			f_1(c_n^{(1)}x_n^{(1)})& \dots & f_1(c_n^{(i-1)}x_n^{(i-1)})& (c_n^{(i)}-c_i)f_1(x_n^{(i)})& f_1(c_{i+1}x_n^{(i+1)}) & \dots & f_1(c_{k+1}x_n^{(k+1)})\\
			\vdots& \ddots &	\vdots& \vdots&		\vdots&	\ddots &\vdots\\
			f_k(c_n^{(1)}x_n^{(1)})& \dots & f_k(c_n^{(i-1)}x_n^{(i-1)}) & (c_n^{(i)}-c_i)f_k(x_n^{(i)})& f_k(c_{i+1}x_n^{(i+1)}) & \dots & f_k(c_{k+1}x_n^{(k+1)})\\
		\end{vmatrix}
	\end{aligned}
\]
for all  $1 \leq i \leq k+1.$ Denote $M= \sup\{\|x_n^{(i)}\|, \|c_n^{(i)}x_n^{(i)}\|, \|c_ix_n^{(i)}\|: 1 \leq i \leq k+1, n \in \mathbb{N}\}.$	Now, for any $1 \leq i \leq k+1,$ by evaluating the determinant $b_n^{(i)}$ along the $i^{th}$ column, we have
\[
	|b_n^{(i)}|  \leq \sum\limits_{s=1}^{k}|(c_n^{(i)}-c_i)f_s(x_n^{(i)}) D_{k-1}[(c_n^{(j)}x_n^{(j)})_{j=1, }^{i-1}, (c_jx_n^{(j)})_{j=i+1}^{k+1}; (f_j)_{j=1, j\neq s}^{k}] |\\
	 \leq |c_n^{(i)}-c_i| M^{k}kk!
\]
and hence $|b_n^{(i)}| \to 0.$ Therefore, $|a_n| \to 0.$\\
	$(2)$: Proof follows by the similar argument involved in the proof of $(1).$
\end{proof}

For any $k \in \mathbb{Z}^{+},$ we define  $\mathcal{S}_k(n)=\{\alpha \subseteq \{1,2, \dots, k\}: \alpha \ \textnormal{contains exactly}\ n \ \textnormal{elements}\}$ for every $n \in \{1, 2, \dots, k\}$  and  $\mathcal{S}_k(0)= \emptyset$. If $n \in \{1,2, \dots, k\}$ and $\alpha \in \mathcal{S}_k(n),$ we denote the elements of $\alpha$ as $\alpha_1, \alpha_2, \dots, \alpha_n,$ where $\alpha_1< \alpha_2 < \dots < \alpha_n.$
\justifying{
\begin{lemma}\label{prop det prop}
	Let $(x_n^{(1)}), (x_n^{(2)}), \dots,  (x_n^{(k+2)})$ be $(k+2)-$sequences in $X$ and $f_1, f_2, \dots, f_{k+1} \in S_{X^*}.$ If $D_k[(x_n^{(\alpha_i)})_{i=1}^{k+1}; (f_{\beta_j})_{j=1}^{k}] \to 0$ for all $\alpha \in \mathcal{S}_{k+2}(k+1)$ and $\beta \in \mathcal{S}_{k+1}(k),$ then $D_{k+1}[(x_n^{(i)})_{i=1}^{k+2};$$ (f_j)_{j=1}^{k+1}] \to 0.$
\end{lemma}
}
\begin{proof}
	Let  $(x_n^{(1)}), (x_n^{(2)}), \dots,  (x_n^{(k+2)})$ be $(k+2)-$sequences in $X$ and $f_1, f_2, \dots, f_{k+1} \in S_{X^*}.$\\
	Case$-(i)$: Suppose $D_1[x_n^{(1)}, x_n^{(i+1)};f_j] \to 0$ for all $1 \leq i, j \leq k+1.$ For every $n \in \mathbb{N},$ using Sylvester's determinant identity \cite[ Page 27]{HoJo2013},  we have $D_{k+1}[(x_n^{(i)})_{i=1}^{k+2}; (f_j)_{j=1}^{k+1}]= det([c_{j,i}^{(n)}]),$ where $c_{j,i}^{(n)}= D_1[x_n^{(1)}, x_n^{(i+1)};f_j]$ for all $1 \leq i,j \leq k+1.$ By evaluating the determinant $det([c_{j,i}^{(n)}])$ along any row and using the assumption,  we have $D_{k+1}[(x_n^{(i)})_{i=1}^{k+2}; (f_j)_{j=1}^{k+1}] \to 0.$\\	
	Case$-(ii)$: Suppose $D_1[x_n^{(1)}, x_n^{(i+1)};f_j]$  does not converge to $ 0$ for some $1 \leq i,j \leq k+1.$ Now, we claim that, there exists $r \in \mathbb{Z}^{+}$ such that $2 \leq r \leq k$ satisfying
	\[
	D_{r-1}[(x_{n}^{(\alpha_i)})_{i=1}^{r}; (f_{\beta_j})_{j=1}^{r-1}]  \nrightarrow 0   \ \mbox{for some} \  \alpha \in \mathcal{S}_{k+2}(r), \beta \in \mathcal{S}_{k+1}(r-1)
	\]
	and
	\[
	D_{r}[(x_n^{(\lambda_i)})_{i=1}^{r+1}; (f_{\mu_j})_{j=1}^{r}] \to 0 \ \mbox{for every} \ \lambda \in \mathcal{S}_{k+2}(r+1), \mu \in \mathcal{S}_{k+1}(r).
	\]
	If $D_{2}[(x_n^{(\lambda_i)})_{i=1}^{3}; (f_{\mu_j})_{j=1}^{2}] \to 0$ for every  $\lambda \in \mathcal{S}_{k+2}(3) \ \mbox{and} \ \mu \in \mathcal{S}_{k+1}(2),$ then by the assumption of Case$-(ii),$ choose $r=2.$ If not, then there exist $ \alpha \in \mathcal{S}_{k+2}(3)$ and  $\beta \in \mathcal{S}_{k+1}(2)$ such that $D_{2}[(x_{n}^{(\alpha_i)})_{i=1}^{3}; (f_{\beta_j})_{j=1}^{2}]  \nrightarrow 0.$ Now if $D_{3}[(x_n^{(\lambda_i)})_{i=1}^{4}; (f_{\mu_j})_{j=1}^{3}] \to 0$ for every  $\lambda \in \mathcal{S}_{k+2}(4) \ \mbox{and} \ \mu \in \mathcal{S}_{k+1}(3),$ then choose $r=3.$ Similarly, proceeding like this and using the hypothesis, the claim holds. Therefore, there exist $r \in \mathbb{Z}^{+}$ with $2 \leq r \leq k,$ $\epsilon>0$ and a subsequence $(n_{m})$ of $(n)$ satisfying
	\[
	\vert D_{r-1}[(x_{n_m}^{(\alpha_i)})_{i=1}^{r}; (f_{\beta_j})_{j=1}^{r-1}] \vert \geq \epsilon  \ \mbox{for some} \  \alpha \in \mathcal{S}_{k+2}(r), \beta \in \mathcal{S}_{k+1}(r-1), \mbox{for all} \ m \in \mathbb{N}
	\]
	and
	\[
	D_{r}[(x_n^{(\lambda_i)})_{i=1}^{r+1}; (f_{\mu_j})_{j=1}^{r}] \to 0 \ \mbox{for every} \ \lambda \in \mathcal{S}_{k+2}(r+1), \mu \in \mathcal{S}_{k+1}(r).
	\]
	Without loss of generality, assume $\alpha=\{1,2, \dots, r\}$ and  $\beta=\{1, 2, \dots, r-1\}.$  For any $m \in \mathbb{N},$ consider
	\[
		\begin{aligned}
			A_{m}= &
			\begin{bmatrix}
				1 & 1 & \dots & 1 \\
				f_1(x_{n_m}^{(1)}) & f_1(x_{n_m}^{(2)}) & \dots & f_1(x_{n_m}^{(k+2)})\\
				\vdots & \vdots & \ddots & \vdots \\
				f_{k+1}(x_{n_m}^{(1)}) & f_{k+1}(x_{n_m}^{(2)}) & \dots & f_{k+1}(x_{n_m}^{(k+2)})\\
			\end{bmatrix}\hspace{-0.1cm}.
		\end{aligned}
	\]
	Now using Sylvester's determinant identity \cite[Page 27]{HoJo2013} for every $m \in \mathbb{N},$ we have
	\[
	det(A_m)D_{r-1}[(x_{n_m}^{(i)})_{i=1}^{r}; (f_{j})_{j=1}^{r-1}]^{k-r+1}= det(B_m),
	\]
	where $B_m= [b_{s,t}^{(r,m)}]_{r+1 \leq s,t \leq k+2}$ and $b_{s,t}^{(r,m)}= D_r[(x_{n_m}^{(i)})_{i=1}^{r},x_{n_m}^{(t)}; (f_j)_{j=1}^{r-1}, f_{s-1}]$ for all $r+1 \leq s, t \leq k+2.$  By evaluating the determinant of $B_m$ along any row and using the claim, we get $det(B_m) \to 0.$  Note that $|D_{k+1}[(x_{n_m}^{(i)})_{i=1}^{k+2}; (f_j)_{j=1}^{k+1}]|= |det(A_m)| \leq \frac{1}{\epsilon^{k-r+1}}|det(B_m)|$ and hence $|D_{k+1}[(x_{n_m}^{(i)})_{i=1}^{k+2}; (f_j)_{j=1}^{k+1}]| \to 0.$ Thus, $|D_{k+1}[(x_{n}^{(i)})_{i=1}^{k+2}; (f_j)_{j=1}^{k+1}]| \to 0.$
\end{proof}
Now, we characterize Schur's property using  $k-$dimensional determinants. The space $X$ is said to have Schur's property, if norm  and weak convergences coincide for sequences in $X.$

\begin{proposition}\label{lem Schur}
	The following statements are equivalent.
	\begin{enumerate}
		\item $X$ has Schur's property.
		\item If $(x_n^{(1)}), (x_n^{(2)}), \dots, (x_n^{(k+1)})$ be $(k+1)-$sequences in $X$ and $D_k[(x_n^{(i)})_{i=1}^{k+1}; (f_j)_{j=1}^{k}] \to 0$ for all $f_1, f_2, \dots, f_{k} \in S_{X^*},$ then $V[(x_n^{(i)})_{i=1}^{k+1}] \to 0.$
	\end{enumerate}
\end{proposition}
\begin{proof}
	$(1) \Rightarrow (2)$: Let $(x_n^{(1)}), (x_n^{(2)}), \dots, (x_n^{(k+1)})$ be $(k+1)-$sequences in $X.$ Assume that $D_k[(x_n^{(i)})_{i=1}^{k+1}; (f_j)_{j=1}^{k}] \to 0$ for all $f_1, f_2, \dots, f_k \in S_{X^*}.$ Observe that, there exist $(k)-$sequences $(f_n^{(1)}), (f_n^{(2)}), \dots, (f_n^{(k)})$ in $S_{X^*}$ such that $V[(x_n^{(i)})_{i=1}^{k+1}] \leq D_k[(x_n^{(i)})_{i=1}^{k+1}; (f_n^{(j)})_{j=1}^{k}] + \frac{1}{n}$
	for all $n \in \mathbb{N}.$ Now, it is enough to show that  $D_k[(x_n^{(i)})_{i=1}^{k+1}; (f_n^{(j)})_{j=1}^{k} ] \to 0.$\\
	\justifying{Step$-(1)$: Fix   $f_2, f_3, \dots, f_{k} \in S_{X^*}.$ For any $f_1 \in S_{X^*},$ by evaluating  $D_k[(x_n^{(i)})_{i=1}^{k+1}; (f_j)_{j=1}^{k}]$ along the $2^{nd}$ row, we have $D_k[(x_n^{(i)})_{i=1}^{k+1}; (f_j)_{j=1}^{k}]=f_1(\sum_{i=1}^{k+1} (-1)^{(2+i)}x_n^{(i)}M_n^{(2,i)}),$ where for any $1 \leq i \leq k+1$ and $n \in \mathbb{N},$ $M_n^{(2,i)}$ denotes the minor of the $(2,i)^{th}$ entry of the determinant $D_k[(x_n^{(i)})_{i=1}^{k+1}; (f_j)_{j=1}^{k}].$  By the assumption,
	$f_1(\sum_{i=1}^{k+1} (-1)^{(2+i)}x_n^{(i)}M_n^{(2,i)}) \to 0 $
	for all $f_1 \in S_{X^*}$ and hence, by $(1),$ we have $\|\sum_{i=1}^{k+1} (-1)^{(2+i)}x_n^{(i)}M_n^{(2,i)}\| \to 0,$ which further implies $f_n^{(1)}(\sum_{i=1}^{k+1} (-1)^{(2+i)}x_n^{(i)}M_n^{(2,i)}) \to 0.$ Therefore, $D_k[(x_n^{(i)})_{i=1}^{k+1}; f_n^{(1)},(f_j)_{j=2}^{k}] \to 0$.}\\
	Step$-(2)$: Fix   $f_3, f_4, \dots, f_{k} \in S_{X^*}.$ For any $f_2 \in S_{X^*},$ by evaluating  $D_k[(x_n^{(i)})_{i=1}^{k+1}; f_n^{(1)},(f_j)_{j=2}^{k}]$ along the $3^{rd}$ row, we have $D_k[(x_n^{(i)})_{i=1}^{k+1}; f_n^{(1)},(f_j)_{j=2}^{k}]=f_2(\sum_{i=1}^{k+1} (-1)^{(3+i)}x_n^{(i)}M_n^{(3,i)}),$ where for any $1 \leq i \leq k+1$ and $n \in \mathbb{N},$  $M_n^{(3,i)}$ denotes the minor of the $(3,i)^{th}$ entry of the determinant $D_k[(x_n^{(i)})_{i=1}^{k+1}; f_n^{(1)},(f_j)_{j=2}^{k}].$
	Now, by using the similar argument involved in Step$-(1),$ we have $D_k[(x_n^{(i)})_{i=1}^{k+1}; f_n^{(1)},f_n^{(2)},(f_j)_{j=3}^{k}] \to 0$.\\
	By repeating the same procedure up to Step$-(k),$ we get $D_k[(x_n^{(i)})_{i=1}^{k+1}; (f_n^{(j)})_{j=1}^{k} ] \to 0.$\\
	$(2) \Rightarrow (1)$: If $X$ is a finite dimensional space, then there is nothing to prove. Let  $X$ be  an  infinite dimensional space and  $(x_n^{(1)})$ be a sequence in $X$ such that $x_n^{(1)} \xrightarrow{w} 0.$ For each $n \in \mathbb{N},$ by Hahn-Banach theorem, there exists $f_n^{(1)} \in S_{X^*}$ such that $f_n^{(1)}(x_n^{(1)})=\|x_n^{(1)}\|.$ Now, for every $n \in \mathbb{N}$ and  $2 \leq i \leq k,$ there exists $x_n^{(i)} \in \cap_{j=1}^{i-1} ker(f_n^{(j)}) \cap S_{X}$ and by  Hahn-Banach theorem, there exists $f_n^{(i)} \in S_{X^*}$ such that $f_n^{(i)}(x_n^{(i)})=1.$  Since $x_n^{(1)} \xrightarrow{w} 0$ and $(x_n^{(i)})$ are bounded sequences for all $2 \leq i \leq k,$ by \Cref{lem samir}, we have $D_k[0,(x_n^{(i)})_{i=1}^{k}; (g_j)_{j=1}^{k}] \to 0$ for all $g_1, g_2, \dots, g_k \in S_{X^*}.$ Therefore, by $(2),$ we have $V[0,(x_n^{(i)})_{i=1}^{k}] \to 0,$ which further implies, $ D_k[0,(x_n^{(i)})_{i=1}^{k}; (f_n^{(j)})_{j=1}^{k}] \to 0.$ Since $D_k[0,(x_n^{(i)})_{i=1}^{k}; (f_n^{(j)})_{j=1}^{k}]=\|x_n^{(1)}\|,$ it follows that $x_n^{(1)} \to 0.$ Hence the proof.
\end{proof}

In the following definition, we introduce two notions called $k-$weakly strong Chebyshevness  and $k-$weakly uniformly strong Chebyshevness which are weaker to the notions $k-$strong Chebyshevness \cite{VeRS2021} and $k-$uniformly strong Chebyshevness \cite{KaVe2018} respectively. These new notions will be used to characterize $k-$WUR, $k-$WMLUR spaces in Section $3$.

\begin{definition}\label{def k-SCh}
	Let $A$ and $B$  be non-empty  subsets of $X,$ $x \in X$ and $k \in \mathbb{Z}^{+}.$ Then we say that $A$ is
	\begin{enumerate}
		\item  $k-$weakly strongly Chebyshev (in short, $k-$$w$SCh) at $x,$ if $A$ is proximinal at $x$ and  for every $\epsilon>0,$ $f_1,f_2,\dots,f_k \in S_{X^*}$ there exists $\delta=\delta(\epsilon,x,(f_j)_{j=1}^{k})>0$ such that $\vert D_k[(x_i)_{i=1}^{k+1};(f_j)_{j=1}^{k}]\vert \leq \epsilon$ whenever $x_1,x_2,\dots,x_{k+1} \in P_A(x,\delta);$
		\item $k-$$w$SCh on $B,$ if $A$ is $k-$$w$SCh at every $x \in B;$
		\item $k-$weakly uniformly strongly Chebyshev (in short, $k-$$w$USCh) on $B,$ if $A$ is proximinal on $B$ and  for every $\epsilon>0,$ $f_1,f_2,\dots,f_k \in S_{X^*}$ there exists $\delta=\delta(\epsilon,(f_j)_{j=1}^{k})>0$ such that $\vert D_k[(x_i)_{i=1}^{k+1};(f_j)_{j=1}^{k}]\vert \leq \epsilon$ whenever $x_1,x_2,\dots,x_{k+1} \in P_A(x,\delta)$ and $x \in B.$
	\end{enumerate}
\end{definition}
The notion  $1-$$w$SCh (respectively, $1-$$w$USCh)  coincides with the notion weakly strongly Chebyshev \cite{BLLN2008, GaTh2022} (respectively, weakly uniformly strongly Chebyshev \cite{GaTh2022}).

Observe that, $A$ is $k-$$w$USCh on $B$  $\Rightarrow$ $A$ is $k-$$w$SCh on $B$ $\Rightarrow$ $A$ is $k-$Chebyshev on $B.$ In  \Cref{eg Ch not wSCh,eg converse Chev2}, we will see that the reverse implications are not necessarily true. Further, $A$ is $k-$USCh (respectively, $k-$SCh) on $B$ $\Rightarrow$ $A$ is  $k-$$w$USCh (respectively, $k-$$w$SCh) on $B,$ in general the converse does not hold (see, \Cref{eg converse Chev}), however, using \Cref{lem Schur}, the converse holds whenever the space has Schur's property.

\begin{example}\label{eg Ch not wSCh}
	 Consider the space $X=(\ell_1, \|\cdot\|_H)$ from \cite[Example 5]{Smit1978a} and $k \in \mathbb{Z}^{+}.$ In \cite{Smit1978a}, it is proved that $X$ is rotund, but not MLUR. Then, by \cite[Theorems 5.1.18 and 5.3.28]{Megg1998}, it follows that  $B_X$ is Chebyshev on $X,$ but not approximatively compact on $X$ (see, \cite[Definition 1.1]{BLLN2008}).  Therefore, by \cite[Lemma 2.8]{VeRS2021}, $B_X$ is not $k-$SCh on $X.$ Since $X$ has Schur's property, we have $B_X$ is not $k-$$w$SCh on $X.$ However, $B_X$ is $k-$Chebyshev on $X.$
\end{example}
The following sequential version of \Cref{def k-SCh} is easy to verify and will be used further.
\begin{proposition}\label{prop kwuch}
	Let $A$ and $B$ be  non-empty subsets of $X$ and $x \in X.$ Then the following statements hold.
	\begin{enumerate}
		\item $A$ is $k-$$w$SCh at $x$ iff  $A$ is proximinal at $x$ and  for any $(k+1)-$sequences $(x_n^{(1)}), (x_n^{(2)}), \dots,$ $ (x_n^{(k+1)})$  in $A$ such that $\|x_n^{(i)}-x\| \to d(x,A)$ for all $1 \leq i \leq k+1,$ it follows that $D_k[(x_n^{(i)})_{i=1}^{k+1};(f_j)_{j=1}^{k}] \to 0$ for all $f_1,f_2,\dots,f_k \in S_{X^*}.$
		\item $A$ is $k-$$w$USCh on $B$ iff $A$ is proximinal on $B$ and for any $(k+1)-$sequences $(x_n^{(1)}), (x_n^{(2)}),$ $ \dots, (x_n^{(k+1)})$  in $A,$ a sequence  $(y_n)$  in $B$ such that $\|x_n^{(i)}-y_n\|-d(y_n,A) \to 0$ for all $1 \leq i \leq k+1,$ it follows that $D_k[(x_n^{(i)})_{i=1}^{k+1};(f_j)_{j=1}^{k}] \to 0$ for all $f_1,f_2,\dots,f_k \in S_{X^*}.$		
	\end{enumerate}
\end{proposition}

Now, we introduce a notion called property $k-$weakly UC which is a generalization of both property $w$UC \cite{GaTh2022} and property $k-$UC \cite{KaVe2018}.
\begin{definition}\label{def property kwuc}
	Let  $A$ and $B$ be  non-empty subsets of $X$ and $k \in \mathbb{Z^+}.$ The pair $(A,B)$ is said to have property $k-$weakly UC (in short, property $k-$$w$UC), if for any $(k+1)-$sequences $(x_n^{(1)}),(x_n^{(2)}), \dots, (x_n^{(k+1)})$ in $A$ and a sequence  $(y_n)$ in $B$ such that $\|x_n^{(i)}-y_n\| \to d(A,B)$  for all $1 \leq i \leq k+1,$ it follows that $D_k[(x_n^{(i)})_{i=1}^{k+1}; (f_j)_{j=1}^{k}] \to 0$ for all $f_1,f_2,\dots,f_k \in S_{X^*}.$
\end{definition}

The property $k-$$w$UC coincides with property $w$UC  \cite{GaTh2022}  for the case $k=1.$ Further, if $(A,B)$ has property $k-$UC, then $(A,B)$ has property $k-$$w$UC. The converse does not hold in general (see, \Cref{eg converse Chev}). However, using \Cref{lem Schur}, the converse holds, whenever the space has Schur's property.

The following result is a consequence of \Cref{prop det prop}. On the other hand, it reveals that if a pair of subsets has property $w$UC, then  it has property $k-$$w$UC and a similar statement holds for the notions $k-$$w$USCh and $k-$$w$SCh.

\begin{proposition}\label{prop wUSCh (k+1)}
	Let $A$ and $B$  be  non-empty  subsets of $X$ and $x \in X.$ Then the following statements hold.
	\begin{enumerate}
		\item If $(A,B)$ has property $k-$$w$UC, then $(A,B)$ has property $(k+1)-$$w$UC.
		\item If $A$ is $k-$$w$USCh on $B,$ then $A$ is $(k+1)-$$w$USCh on $B.$
		\item If $A$ is $k-$$w$SCh at $x,$ then $A$ is $(k+1)-$$w$SCh at $x.$
	\end{enumerate}
\end{proposition}
\begin{proof}
	$(1)$: Let $(x_n^{(1)}), (x_n^{(2)}), \dots, (x_n^{(k+2)})$ be $(k+2)-$sequences in $A,$ $(y_n)$ be a sequence  in $B$ such that $\|x_n^{(i)}-y_n\|\to d(A, B)$ for all $1 \leq i \leq k+2$ and $f_1, f_2, \dots, f_{k+1} \in S_{X^*}.$  By assumption, it follows that   $D_k[(x_n^{(\alpha_i)})_{i=1}^{k+1}; (f_{\beta{j}})_{j=1}^{k}] \to 0$ for all $\alpha \in \mathcal{S}_{k+2}(k+1)$ and $\beta \in \mathcal{S}_{k+1}(k).$ Hence, by \Cref{prop det prop}, we have $D_{k+1}[(x_n^{(i)})_{i=1}^{k+2}; (f_j)_{j=1}^{k+1}] \to 0.$ Thus, $(A,B)$ has property $(k+1)-$$w$UC.
	
	 The proofs of $(2)$ and $(3)$ follow in the similar lines of proof of $(1).$	
\end{proof}
We remark that the converses of the statements of \Cref{prop wUSCh (k+1)} need not be true for any $k \in \mathbb{Z}^{+}$ (see, \Cref{eg k+1 not k USCh}).

In the following  proposition and remark, we present some relations among the notions $k-$$w$SCh, $k-$$w$USCh and property $k-$$w$UC.
 \begin{proposition}\label{kwusch kwuc}
 	Let $A$ and $B$ be non-empty subsets of $X.$ Then the following statements hold.
 	\begin{enumerate}
 		\item If $A$ is $k-$$w$USCh on $B,$ then $(A,B)$ has property $k-$$w$UC.
 		\item If $(A,B)$ has property $k-$$w$UC, then $A$ is $k-$$w$USCh on $B_0,$ where $B_0=\{y \in B: \|x-y\|=d(A,B)$ for some $x \in A \}.$
 		\end{enumerate}
 		\end{proposition}
 \begin{proof} $(1)$: Let $(x_n^{(1)}), (x_n^{(2)}), \dots, (x_n^{(k+1)})$ be $(k+1)-$sequences in $A$ and  $(y_n)$ be a sequence in $B$ such that $\|x_n^{(i)}-y_n\| \to d(A,B)$ for all $1 \leq i \leq k+1.$ Since, for any $1 \leq i \leq k+1,$
 	\[
 	0 \leq \|x_n^{(i)}-y_n\|-d(y_n,A) \leq \|x_n^{(i)}-y_n\|-d(A,B),
 	\]
 	 we have $\|x_n^{(i)}-y_n\|-d(y_n,A) \to 0.$ Thus, by assumption, it follows that
  $D_k[(x_n^{(i)})_{i=1}^{k+1};(f_j)_{j=1}^{k}] \to 0$ for all $f_1,f_2,\dots,f_k \in S_{X^*}.$ Hence, $(A,B)$ has property $k-$$w$UC.\\
  $(2)$: Clearly, $A$ is proximinal on $B_0.$ Let  $(x_n^{(1)}),(x_n^{(2)}),\dots,(x_n^{(k+1)})$ be $(k+1)-$sequences in A and  $(y_n)$ be a sequence in $B_0$ such that $\|x_n^{(i)}-y_n\|- d(y_n,A) \to 0$ for all $1 \leq i \leq k+1.$ Since $d(y_n,A)= d(A,B)$ for all $n \in \mathbb{N},$ we have $\|x_n^{(i)}-y_n\| \to d(A,B)$ for all $1 \leq i \leq k+1.$  Therefore, by assumption, it follows that $D_k[(x_n^{(i)})_{i=1}^{k+1};(f_j)_{j=1}^{k}] \to 0$ for all $f_1,f_2,\dots,f_k \in S_{X^*}.$ Thus, $A$ is $k-$$w$USCh on $B_0.$
 \end{proof}

 \begin{remark}\label{rem boundely compact}
 	Let $A$ be a non-empty bounded subset of $X$ and  $B$ be a non-empty boundedly compact subset of $X.$ If $A$ is $k-$$w$SCh on $B,$ then $(A,B)$ has property $k-$$w$UC.
 \end{remark}
The next example shows that  the converse of the statements of \Cref{kwusch kwuc} and \Cref{rem boundely compact} need not be true. In particular, property $k-$$w$UC of the pair $(A,B)$ is not sufficient for the proximinality of $A$ on $B.$%
\begin{example}\label{eg converse}\
	\begin{enumerate}
		\item Let $k \in \mathbb{Z}^{+}$ and $M=(c_0, \|\cdot\|_{\infty}).$ By \cite[Chapter II, Corollary 6.9]{DeGZ1993}, $M$ admits an equivalent norm (say, $\|\cdot\|_r$) such that $X=(c_0, \|\cdot\|_r)$ is WUR. Since $X$ is not reflexive, there exists a  subspace $Y$ of $X$ such that $Y$ is not  proximinal at some $x \in X.$ However, by \cite[Theorem 4.6]{GaTh2022}, $(Y, \{x\})$ has property $w$UC and hence, by \Cref{prop wUSCh (k+1)}, it has property $k-$$w$UC.
		\item Let $k \in \mathbb{Z}^{+}$ and $X= (\mathbb{R}^{k+1}, \|\cdot\|_{\infty}).$ Consider $A= B_X$ and $B= 3S_X \cup \{2(\sum_{i=1}^{k+1}e_i)\}.$ It is easy to prove that $(A,B)$ has property $w$UC and hence, by \Cref{prop wUSCh (k+1)}, it has property $k-$$w$UC. However, it is clear that $A$ is not $k-$Chebyshev at $3e_1 \in B.$ Thus, $A$ is not $k-$$w$USCh on $B.$
		\item Let $k \in \mathbb{Z}^{+},$ $X= (\mathbb{R}^{k+1}, \|\cdot\|_2) \oplus_{\infty} \mathbb{R}$ and $Y =(\mathbb{R}^{k+1}, \|\cdot\|_2) \oplus_{\infty} \{0\}$ be the subspace of $X.$ Consider  $A=B_Y$ and $B=\{(2e_1,0)\} \cup \{(0, 1+\frac{1}{n}): n \in \mathbb{N}\}.$ Observe that $d(A,B)=1$ and $B_0=\{y \in B: \|x-y\|=d(A,B) \text{ for some } x \in A \}=\{(2e_1,0)\}.$ Clearly, $A$ is $k-$$w$USCh on $B_0.$ For all $n \in \mathbb{N}$ and  $1 \leq i \leq k+1,$ define $x_n^{(i)}=(e_i,0)$ and $y_n=(0, 1+\frac{1}{n}).$ Therefore, $\|x_n^{(i)}-y_n\| \to 1$ for all $1 \leq i \leq k+1,$ but by \Cref{remark vol}, there exists $g_1, g_2, \dots, g_{k} \in S_{X^*}$ such that  $D_k[(x_n^{(i)})_{i=1}^{k+1}; (g_j)_{j=1}^{k}]= \epsilon$ for some $\epsilon>0.$  Thus, $(A,B)$ does not have  property $k-$$w$UC.
	\end{enumerate}
	\end{example}

The  proof of the subsequent result follows in similar lines of the proof  of \cite[Theorem 2.12]{GaTh2023}.
\begin{theorem}\label{thrm wUSCh}\
	The following statements hold.
	\begin{enumerate}
		\item If $B_X$ is $k-$$w$USCh (respectively, $k-$$w$SCh) on $rS_X$ for some $r\in (1, \infty),$ then $B_X$ is $k-$$w$USCh (respectively, $k-$$w$SCh) on $tS_X$ for every $t \in (1, \infty).$
		\item If $S_X$ is $k-$$w$USCh (respectively, $k-$$w$SCh) on $rS_X$ for some $r \in (1, \infty),$ then $S_X$ is $k-$$w$USCh (respectively, $k-$$w$SCh) on $tS_X$ for every $t \in (1, \infty).$
		\item If $S_X$ is $k-$$w$USCh (respectively, $k-$$w$SCh) on $rS_X$ for some $r \in (0,1),$ then $S_X$ is $k-$$w$USCh (respectively, $k-$$w$SCh) on $tS_X$ for every $t \in (0,\infty).$	
	\end{enumerate}	
\end{theorem}	
Now, we present a characterization of $k-$rotund spaces in terms of $k-$rotundity of the quotient spaces.
\begin{theorem}\label{thrm quotiet k-rotund}
	Let $\alpha, \beta \in \mathbb{Z}^{+}$ and $X$ be a Banach space satisfying  $dim(X) \geq k+2,$ $1 \leq \alpha \leq dim(X)-(k+1)$ and  $k+1 \leq \beta \leq  dim(X)-1.$ Consider the following statements.
	\begin{enumerate}
		\item $X$ is $k-$rotund.
		\item $X/M$ is $k-$rotund, whenever $M$ is a proximinal subspace of $X.$
		\item $X/F$ is $k-$rotund, whenever $F$ is a subspace of $X$ with $dim(F)= \alpha.$
		\item $X/Y$ is $k-$rotund, whenever $Y$ is a proximinal subspace of $X$ with $codim(Y)= \beta.$
	\end{enumerate}
Then $(1) \Leftrightarrow (2) \Leftrightarrow (3) \Rightarrow (4).$ Further, if $X$ is reflexive, then all the statements are equivalent.
\end{theorem}
\begin{proof}
	$(1) \Rightarrow (2)$: Let $M$ be a proximinal subspace of $X.$ Let $x_1+M, x_2+M, \dots, x_{k+1}+M \in S_{X/M}$ with $ \Vert \sum_{i=1}^{k+1}(x_i+M) \Vert =k+1.$ Since $M$ is proximinal on $X,$ for every $1 \leq i \leq k+1$ there exists $y_i \in M$ such that $\|x_i-y_i\|= d(x_i,M)=1.$ Note that
	\[
	k+1=  \left \Vert \sum\limits_{i=1}^{k+1} (x_i+M) \right \Vert = d\left(\sum\limits_{i=1}^{k+1}x_i,M\right) \leq  \left \Vert \sum\limits_{i=1}^{k+1}x_i - \sum\limits_{i=1}^{k+1}y_i \right\Vert \leq k+1,
	\]
	which implies $\|\sum_{i=1}^{k+1}(x_i -y_i)\|= k+1.$ Therefore, by $(1),$ we have $V[(x_i-y_i)_{i=1}^{k+1}]=0.$ Using \Cref{remark vol}, it is easy to verify that $V[(x_i+M)_{i=1}^{k+1}] \leq V[(x_i-y_i)_{i=1}^{k+1}]$ and hence $V[(x_i+M)_{i=1}^{k+1}] = 0.$ Thus, $X/M$ is $k-$rotund.\\
	$(2) \Rightarrow (3)$: Obvious.\\
	$(3) \Rightarrow (1)$: Suppose there exist $x_1, x_2, \dots, x_{k+1} \in S_{X}$ with $\|\sum_{i=1}^{k+1}x_i\|=k+1$ such that $V[(x_i)_{i=1}^{k+1}]>0.$ By Hahn-Banach theorem, there exists  $f \in S_{X^*}$ such that $f(\sum_{i=1}^{k+1}x_i)= \|\sum_{i=1}^{k+1}x_i\|.$  Therefore, $f(x_i)=1$ for all $1 \leq i \leq k+1.$ Choose a subspace $F$ such that $F\subseteq ker(f),$ $F \cap span\{x_i- x_{k+1}: 1 \leq i \leq k\}=\{0\}$ and $dim(F)=\alpha.$ Hence, by Ascoli's formula, for all $1 \leq i \leq k+1,$ we have
	\[
	1= |f(x_i)|=d(x_i, ker(f)) \leq d(x_i,F) \leq \|x_i\|=1.
	\]
	Therefore, $\|x_i+F\|=1$ for all $1 \leq i \leq k+1.$ Similarly, we have $\|\sum_{i=1}^{k+1}(x_i+F)\|=k+1.$ By $(3),$ we get $V[(x_i+F)_{i=1}^{k+1}]=0.$ Thus, by \Cref{remark vol}, there exist $\lambda_1, \lambda_2, \dots ,\lambda_{k} \in \mathbb{R}$ such that $\lambda_k =1$ and $ \sum_{i=1}^{k} \lambda_i(x_i-x_{k+1}+F)=0+F.$ Observe that  $\sum_{i=1}^{k}\lambda_i(x_i-x_{k+1}) \in F.$ Therefore $\sum_{i=1}^{k}\lambda_i(x_i-x_{k+1}) =0,$ which  implies $V[(x_i)_{i=1}^{k+1}]=0.$ This is a contradiction.\\
	$(2) \Rightarrow (4)$: Obvious.
	
	 Let $X$ be a reflexive space. Suppose there exist $x_1, x_2, \dots, x_{k+1} \in S_{X}$ with $\|\sum_{i=1}^{k+1}x_i\|=k+1$ such that $V[(x_i)_{i=1}^{k+1}]>0.$ By Hahn-Banach theorem, there exists  $f \in S_{X^*}$ such that $f(\sum_{i=1}^{k+1}x_i)= \|\sum_{i=1}^{k+1}x_i\|.$ Therefore, $f(x_i)=1$ for all $1 \leq i \leq k+1.$ Choose a subspace $Y$ such that $Y\subseteq ker(f),$ $codim(Y)=\beta$ and  $Y \cap span\{x_i- x_{k+1}: 1 \leq i \leq k\}=\{0\}$.  Since $Y$ is proximinal on $X,$ by replacing $F$ by $Y$ in the proof of $(3) \Rightarrow (1)$ and repeating the argument involved in the proof, we get a contradiction. Hence the proof.
\end{proof}

As a consequence of \Cref{eg quotient k-rotund}, we  observe that the implication $(4) \Rightarrow (1)$ of \Cref{thrm quotiet k-rotund} need not be true in general, for any $k \in \mathbb{Z}^{+}.$
\begin{example}\label{eg quotient k-rotund}
	Let $k \in \mathbb{Z}^{+}$ and  $X= M \oplus_{1} (\mathbb{R}^{k}, \|\cdot\|_{1}),$ where $M$ is the Read's space \cite{Read2018}. Clearly, $X$ is not $k-$rotund. Let $Y$ be any subspace of $X$ with $codim(Y)=k+2.$ Since any finite co-dimensional subspace of $M$ with co-dimension greater than one is not proximinal on $M,$ by \cite[Corollary 4.2]{BLLN2008}, it follows that $Y$ is not proximinal on $X.$ Therefore, $X$ does not have any proximinal subspace of co-dimension $k+2.$
\end{example}

\section{Characterizations of $k-$WUR, $k-$WLUR and $k-$WMLUR}
In this section, we introduce and study two notions called $k-$weakly uniform rotundity and $k-$weakly locally uniform rotundity. We present a few characterizations of $k-$WUR, $k-$WLUR and $k-$WMLUR in terms of the notions discussed in Section $2$.
		
		\begin{definition} \label{def k-wur}
			Let $k \in \mathbb{Z^+}.$ A space $X$ is said to be
			\begin{enumerate}
				\item $k-$weakly uniformly rotund (in short, $k-$WUR), if for every $\epsilon >0$ and $f_1,f_2,\dots,f_k \in S_{X^*}$,
				\[
				\delta^{k}_{X}(\epsilon,(f_j)_{j=1}^{k}) \coloneqq  \inf\left\{ \{1\} \cup \left\{1-\dfrac{1}{k+1}\left \Vert \sum\limits_{i=1}^{k+1}x_i \right \Vert:
				\begin{array}{l}
					
					x_1, x_2, \dots,x_{k+1}\in S_X,\\
					| D_k[(x_i)_{i=1}^{k+1}; (f_j)_{j=1}^{k}]|\geq \epsilon
				\end{array}
				\right\}
				\right\} >0;
				\]
				\item $k-$weakly locally uniformly rotund (in short, $k-$WLUR) at $x \in S_X$, if for every
				$\epsilon >0$ and  $f_1,f_2,\dots,f_k \in S_{X^*}$,
				\[
				\delta^{k}_{X}(\epsilon, x, (f_j)_{j=1}^{k}) \coloneqq  \inf\left\{ \{1\} \cup \left\{1-\dfrac{1}{k+1}\left \Vert x+ \sum\limits_{i=1}^{k}x_i \right \Vert:
				\begin{array}{l}
					
					x_1, x_2, \dots,x_{k}\in S_X,\\
					| D_k[x, (x_i)_{i=1}^{k}; (f_j)_{j=1}^{k}]|\geq \epsilon
				\end{array}
				\right\}
				\right\} >0.
				\]
				We say $X$ is $k-$weakly locally uniformly rotund (in short, $k-$WLUR), if $X$ is $k-$WLUR at every $x \in S_X$.
			\end{enumerate}
		\end{definition}
		Clearly, the notion $1-$WUR (respectively, $1-$WLUR) coincide with the notion WUR (respectively, WLUR). The equivalent sequential formulation  of the notions $k-$WUR and $k-$WLUR given in the following results are useful to prove our results.
\begin{proposition}\label{equkwur}
	The  following statements are equivalent.
	\begin{enumerate}
		\item $X$ is $k-$WUR.
		\item If $(x_n^{(1)}),(x_n^{(2)}), \dots, (x_n^{(k+1)})$ are $(k+1)-$sequences in $S_X$ such that $\frac{1}{k+1}\|\sum_{i=1}^{k+1}x_n^{(i)}\| \to 1,$ then $D_k[(x_n^{(i)})_{i=1}^{k+1};(f_j)_{j=1}^{k}] \to 0$ for all  $f_1,f_2,\dots,f_k \in S_{X^*}.$
		\item If $(x_n^{(1)}),(x_n^{(2)}), \dots, (x_n^{(k+1)})$ are $(k+1)-$sequences in $B_X$ such that $\frac{1}{k+1}\|\sum_{i=1}^{k+1}x_n^{(i)}\| \to 1,$ then $D_k[(x_n^{(i)})_{i=1}^{k+1};(f_j)_{j=1}^{k}] \to 0$  for all $f_1,f_2,\dots,f_k \in S_{X^*}.$
		\item If $(x_n^{(1)}),(x_n^{(2)}), \dots, (x_n^{(k+1)})$ are $(k+1)-$sequences in $X$ such that $\|x_n^{(i)}\| \to 1$ for all $1 \leq i \leq k+1$ and $\frac{1}{k+1}\|\sum_{i=1}^{k+1}x_n^{(i)}\| \to 1,$ then $D_k[(x_n^{(i)})_{i=1}^{k+1};(f_j)_{j=1}^{k}] \to 0$  for all $f_1,f_2,\dots,f_k \in S_{X^*}.$
	\end{enumerate}
\end{proposition}
\begin{proof} $(1)\Leftrightarrow(2)$: These implications follow from the \Cref{def k-wur}. \\
	$(2) \Rightarrow (4)$: Let  $(x_n^{(1)}),(x_n^{(2)}), \dots, (x_n^{(k+1)})$ be $(k+1)-$sequences in $X$ with $\|x_n^{(i)}\| \to 1$ for all $1 \leq i \leq k+1$ and  $\frac{1}{k+1}\| \sum_{i=1}^{k+1}x_n^{(i)} \| \to 1.$ Let $f_1,f_2,\dots,f_k \in S_{X^*}.$ For all $n \in \mathbb{N}$ and $1 \leq i \leq k+1,$ define $\overline{x}_n^{(i)}= \frac {x_n^{(i)}}{\|x_n^{(i)}\|}.$ Since
	\[
1  \geq  \frac{1}{k+1}\left \Vert \sum\limits_{i=1}^{k+1}\overline{x}_n^{(i)} \right \Vert \geq  \frac{1}{k+1}\left \Vert  \sum\limits_{i=1}^{k+1}x_n^{(i)} \right \Vert-  \frac{1}{k+1}\sum\limits_{i=1}^{k+1} \left \Vert\overline{x}_n^{(i)}-x_n^{(i)} \right \Vert,		
	\]
it follows that  $ \frac{1}{k+1}\|\sum_{i=1}^{k+1}\overline{x}_n^{(i)}\| \to 1.$ Thus, by assumption, $D_k[(\overline{x}_n^{(i)})_{i=1}^{k+1};(f_j)_{j=1}^{k}] \to 0.$ Further, using \Cref{lem samir}, we have $D_k[(x_n^{(i)})_{i=1}^{k+1};(f_j)_{j=1}^{k}] \to 0.$\\
	$(4) \Rightarrow (3)$: Let  $(x_n^{(1)}),(x_n^{(2)}), \dots, (x_n^{(k+1)})$ be $(k+1)-$sequences in $B_X$ such that  $\frac{1}{k+1}\|\sum_{i=1}^{k+1} x_n^{(i)}\| \to 1.$ Note that for any $1 \leq i \leq k+1,$ we have
	\[
	\frac{1}{k+1}\left \Vert \sum\limits_{i=1}^{k+1}x_n^{(i)} \right \Vert \leq \frac{\|x_n^{(i)}\|}{k+1}+ \frac{k}{k+1} \leq 1
	\]
	and hence $\|x_n^{(i)}\| \to 1.$ Thus, by $(4),$ $D_k[(x_n^{(i)})_{i=1}^{k+1};(f_j)_{j=1}^{k}] \to 0$ for all $f_1,f_2,\dots, f_k \in S_{X^*}.$ \\
	$(3) \Rightarrow (2)$: Obvious.
\end{proof}
The proof of the following corollary is similar to the proof of \Cref{equkwur}.
\begin{corollary}
	Let $x \in S_X$. Then the  following statements are equivalent.
	\begin{enumerate}
		\item $X$ is $k-$WLUR at $x$.
		\item If $(x_n^{(1)}),(x_n^{(2)}), \dots, (x_n^{(k)})$ are $(k)-$sequences in $S_X$ such that $\frac{1}{k+1}\|x+\sum_{i=1}^{k}x_n^{(i)}\| \to 1,$ then $D_k[x, (x_n^{(i)})_{i=1}^{k};(f_j)_{j=1}^{k}] \to 0$ for all  $f_1,f_2,\dots,f_k \in S_{X^*}.$
		\item If $(x_n^{(1)}),(x_n^{(2)}), \dots, (x_n^{(k)})$ are $(k)-$sequences in $B_X$ such that $\frac{1}{k+1}\|x+\sum_{i=1}^{k}x_n^{(i)}\| \to 1,$ then $D_k[x, (x_n^{(i)})_{i=1}^{k};(f_j)_{j=1}^{k}] \to 0$  for all $f_1,f_2,\dots,f_k \in S_{X^*}.$
		\item If $(x_n^{(1)}),(x_n^{(2)}), \dots, (x_n^{(k)})$ are $(k)-$sequences in $X$ such that $\|x_n^{(i)}\| \to 1$ for all $1 \leq i \leq k$ and $\frac{1}{k+1}\|x+\sum_{i=1}^{k}x_n^{(i)}\| \to 1,$ then $D_k[x, (x_n^{(i)})_{i=1}^{k};(f_j)_{j=1}^{k}] \to 0$  for all $f_1,f_2,\dots,f_k \in S_{X^*}.$
	\end{enumerate}
\end{corollary}

 It is easy to verify that the observations given in the following remark hold.
\begin{remark}\label{rem KWUR}\
	\begin{enumerate}
		\item 	From the definitions, it follows that $k-$UR  $\Rightarrow$ $k-$WUR $\Rightarrow$ $k-$WLUR  $\Rightarrow$ $k-$rotund. Further, $k-$LUR $\Rightarrow$ $k-$WLUR.
		\item In general, none of the  implications given in $(1)$ can be reversed  (see, \Cref{eg kWUR,eg kWUR not KMLUR}). However, if the space is finite dimensional, then all the notions in $(1)$ coincide.
		\item There is no relation between the notion $k-$WUR and any of the  notions $k-$LUR, $k-$MLUR, $k-$strongly rotund (see, \Cref{eg kWUR,eg kWUR not KMLUR}). Also, there is no relation between the notion $k-$WLUR and any of the notions $k-$MLUR, $k-$strongly rotund  (see, \Cref{eg kWUR not KMLUR,eg kWUR}).
	\end{enumerate}
\end{remark}

 The following result is an outcome of \Cref{prop det prop}, wherein we show that if a space is WUR (respectively, WLUR), then it is $k-$WUR (respectively, $k-$WLUR) for any $k \in \mathbb{Z}^+.$
\begin{proposition}\label{prop k implies k+1}
Let $x \in S_X.$ Then the following statements hold.
	\begin{enumerate}
		\item If $X$ is $k-$WUR, then $X$ is $(k+1)-$WUR.
		\item If $X$ is $k-$WLUR at $x,$ then $X$ is $(k+1)-$WLUR at $x.$
	\end{enumerate}
\end{proposition}
\begin{proof} $(1)$: Let $(x_n^{(1)}), (x_n^{(2)}), \dots, (x_n^{(k+2)})$ be $(k+2)-$sequences in $S_X$ such that $\|\sum_{i=1}^{k+2}x_n^{(i)}\| \to k+2$ and $f_1, f_2, \dots, f_{k+1} \in S_{X^*}.$ Since for any $1 \leq j \leq k+2,$ we have
\[
	\left\Vert\sum\limits_{i=1}^{k+2}x_n^{(i)}\right\Vert- 1=\left\Vert \sum\limits_{i=1}^{k+2}x_n^{(i)}\right\Vert- \left\Vert x_n^{(j)}\right\Vert \leq \left\Vert \sum\limits_{i=1, i \neq j}^{k+2}x_n^{(i)}\right\Vert \leq k+1,
\]
	which implies   $\|\sum_{i=1, i \neq j}^{k+2}x_n^{(i)}\| \to k+1.$  By assumption, we have $D_{k}[(x_n^{(\alpha_i)})_{i=1}^{k+1}; (f_{\beta_j})_{j=1}^{k}] \to 0$ for all $\alpha \in \mathcal{S}_{k+2}(k+1)$ and $\beta \in \mathcal{S}_{k+1}(k).$  Therefore, by \Cref{prop det prop},  $D_{k+1}[(x_n^{(i)})_{i=1}^{k+2};(f_j)_{j=1}^{k+1}] \to 0.$ Thus, $X$ is $(k+1)-$WUR.\\
	$(2)$: Let $x \in S_X,$ $(x_n^{(1)}), (x_n^{(2)}), \dots, (x_n^{(k+1)})$ be $(k+1)-$sequences in $S_X$ with $\|x+\sum_{i=1}^{k+1}x_n^{(i)}\| \to k+2$ and  $f_1, f_2, \dots, f_{k+1} \in S_{X^*}.$ Note that, $\|x+\sum_{i=1}^{k}x_n^{(\alpha_i)}\| \to k+1$ for all $\alpha \in \mathcal{S}_{k+1}(k).$ Since $X$ is $k-$WLUR at $x,$ it follows that  $|D_{k}[x,(x_n^{(\alpha_i)})_{i=1}^{k}; (f_{\beta_j})_{j=1}^{k}]| \to 0$ for all $\alpha, \beta \in \mathcal{S}_{k+1}(k).$  Now, as a result of \cite[Lemma 2]{Suya2000}, for any $\beta \in \mathcal{S}_{k+1}(k),$ we have
	\[
|D_{k}[(x_n^{(i)})_{i=1}^{k+1}; (f_{\beta_j})_{j=1}^{k}]| \leq  \sum_{\alpha \in \mathcal{S}_{k+1}(k)}|D_{k}[x, (x_n^{(\alpha_i)})_{i=1}^{k}; (f_{\beta_j})_{j=1}^{k}]|,
	\]
	 which  implies $ |D_{k}[(x_n^{(i)})_{i=1}^{k+1}; (f_{\beta_j})_{j=1}^{k}]| \to 0.$  Thus, by \Cref{prop det prop},
 $D_{k+1}[x,(x_n^{(i)})_{i=1}^{k+1}; (f_j)_{j=1}^{k+1}] \to 0.$ Hence, $X$ is $(k+1)-$WLUR at $x.$
\end{proof}

  The subsequent example shows that the converses of the statements of \Cref{prop k implies k+1}, need not be true for any $k \in \mathbb{Z}^{+}.$ Further, we will see in \Cref{eg k+1 WUR} that there exists a strongly rotund space which is $(k+1)-$WUR, but not $k-$WLUR.

\begin{example}\label{eg k not k-1 WUR}
 Let $k \in \mathbb{Z}^{+},$ $k \geq 2$ and  $i_1 < i_2 < \dots< i_k.$ For each $x=(x_1, x_2, \dots)$ in $l_2,$ define
		\[
		\|x\|_{i_1,i_2, \dots, i_k}^{2}= \left(\sum\limits_{j=1}^{k}\vert x_{i_j} \vert \right)^2+ \sum\limits_{i \neq i_1, i_2, \dots, i_k} x_i^2.
		\]
		Let $X= (l_2, \| \cdot\|_{i_1, i_2, \dots, i_k}).$ In \cite[Example 2]{LiYu1985}, it is proved that $X$ is $k-$UR, but not $(k-1)-$rotund. Thus, $X$ is $k-$WUR, but not $(k-1)-$WLUR.
\end{example}
As noted in \Cref{rem KWUR}, now we  provide an example.
\begin{example}\label{eg kWUR not KMLUR}
	Consider the space	$X= (\ell_2,\| \cdot \|_W)$ from $\textnormal{\cite[Example 2]{Smit1978a}}$ and $k \in \mathbb{Z}^{+}.$ In \cite{Smit1978a}, it is proved that $X$ is WUR, but not MLUR and it does not have the Kadets-Klee property (see, \cite[Definition 2.5.26]{Megg1998}). From \cite[Theorems 5.1.18 and 5.3.28]{Megg1998}, it follows that  $B_X$ is Chebyshev on $X,$ but not approximatively compact on $X.$ Therefore, by \cite[Lemma 2.8]{VeRS2021}, $B_X$ is not $k-$SCh on $X.$ Thus, by \cite[Theorem 2.6]{LiZZ2018}, $X$ is not $k-$MLUR. Observe that $X$ is  not $k-$strongly rotund. However, by \Cref{prop k implies k+1}, $X$ is $k-$WUR.
\end{example}
 Now, we present some sequential characterizations of $k-$WUR in terms of an uniform version of  $k-$WMLUR.

  \begin{definition}\textnormal{\cite{XiLi2004}}\label{def KWLUR}
 	Let $k \in \mathbb{Z^+}.$ A space $X$ is said to be $k-$WMLUR, if for any $(k+1)-$sequences $(x_n^{(1)}),(x_n^{(2)}),$ $ \dots, (x_n^{(k+1)})$ in $S_X$ and $x \in S_X$ with $\|(k+1)x-\sum_{i=1}^{k+1}x_n^{(i)}\| \to 0,$ it follows that $D_k[(x_n^{(i)})_{i=1}^{k+1};(f_j)_{j=1}^{k}] \to 0$ for all  $f_1,f_2,\dots,f_k \in S_{X^*}.$
 \end{definition}
It is easy to verify from the definitions that $k-$WLUR $\Rightarrow$ $k-$WMLUR $\Rightarrow$ $k-$rotund. However, none of the implications can be reversed in general (see, \Cref{eg kWMLUR not kR,eg kWUR}).
\begin{theorem}\label{thrm WUR MLUR}
	The following statements are equivalent.
	\begin{enumerate}
		\item $X$ is $k-$WUR.
		\item If $(x_n^{(1)}), (x_n^{(2)}), \dots, (x_n^{(k+2)})$ are $(k+2)-$sequences in $S_X$ such that $\|(k+1)x_n^{(k+2)}-\sum_{i=1}^{k+1}x_n^{(i)}\| \to 0,$ then $D_k[(x_n^{(\alpha_i)})_{i=1}^{k+1}; (f_j)_{j=1}^{k}] \to 0$ for all $f_1, f_2, \dots, f_k \in S_{X^*}$ and $\alpha \in \mathcal{S}_{k+2}(k+1).$
		\item If $(x_n^{(1)}), (x_n^{(2)}), \dots, (x_n^{(k+2)})$ are $(k+2)-$sequences in $S_X$ such that $\|(k+1)x_n^{(k+2)}-\sum_{i=1}^{k+1}x_n^{(i)}\| \to 0,$ then $D_k[(x_n^{(\alpha_i)})_{i=1}^{k+1}; (f_j)_{j=1}^{k}] \to 0$ for all $f_1, f_2, \dots, f_k \in S_{X^*}$ and  for some  $\alpha \in \mathcal{S}_{k+2}(k+1).$
	\end{enumerate}
\end{theorem}

\begin{proof}
	$(1) \Rightarrow (2)$: Let $(x_n^{(1)}), (x_n^{(2)}), \dots, (x_n^{(k+2)})$ be $(k+2)-$sequences in $S_X$ with $\|(k+1)x_n^{(k+2)}-\sum_{i=1}^{k+1}x_n^{(i)}\| \to 0$ and $f_1, f_2, \dots, f_{k} \in S_{X^*}.$ Observe that $\| \sum_{i=1}^{k+1}x_n^{(i)}\| \to k+1.$ Thus, by $(1),$ we get $D_k[(x_n^{(i)})_{i=1}^{k+1}; (f_j)_{j=1}^{k}] \to 0.$ Now, it is enough to show that $D_k[x_n^{(k+2)},(x_n^{(\beta_i)})_{i=1}^{k}; (f_j)_{j=1}^{k}] \to 0$  for all $\beta \in \mathcal{S}_{k+1}(k).$  For every $n \in \mathbb{N}$, consider $y_n=\frac{1}{k+1}\sum_{i=1}^{k+1}x_n^{(i)}$ and let $\beta \in \mathcal{S}_{k+1}(k)$.  Note that for any $n \in \mathbb{N},$ we have
	 $
	 |D_k[(x_n^{(i)})_{i=1}^{k+1}; (f_j)_{j=1}^{k}]|= (k+1) |D_k[y_n, (x_n^{(\beta_i)})_{i=1}^{k}; (f_j)_{j=1}^{k}]|,
	 $
	 which  implies  $D_k[y_n, (x_n^{(\beta_i)})_{i=1}^{k}; (f_j)_{j=1}^{k}] \to 0.$ Since $y_n- x_n^{(k+2)} \to 0,$ it follows from \Cref{lem samir} that
	 	 \[
|D_k[y_n-x_n^{(k+2)}+x_n^{(k+2)},(x_n^{(\beta_i)})_{i=1}^{k}; (f_j)_{j=1}^{k}] -D_k[x_n^{(k+2)},(x_n^{(\beta_i)})_{i=1}^{k}; (f_j)_{j=1}^{k}]| \to 0.
\]
Therefore, $D_k[x_n^{(k+2)},(x_n^{(\beta_i)})_{i=1}^{k}; (f_j)_{j=1}^{k}] \to 0.$ \\
$(2) \Rightarrow (3)$: Obvious.\\
$(3) \Rightarrow (1)$: Let $(x_n^{(1)}), (x_n^{(2)}), \dots, (x_n^{(k+1)})$ be $(k+1)-$sequences in $S_X$ such that $\|\sum_{i=1}^{k+1}x_n^{(i)}\| \to k+1$ and $f_1, f_2, \dots, f_{k}$ in $ S_{X^*}.$ For each $n \in \mathbb{N},$ define $x_n^{(k+2)} = \frac{\sum_{i=1}^{k+1}x_n^{(i)}}{\|\sum_{i=1}^{k+1}x_n^{(i)}\|}.$ Since $\| (k+1)x_n^{(k+2)} - \sum_{i=1}^{k+1}x_n^{(i)} \| \to 0$, by $(3),$ $D_k[(x_n^{(\alpha_i)})_{i=1}^{k+1}; (f_j)_{j=1}^{k}] \to 0$ for some $\alpha \in \mathcal{S}_{k+2}(k+1).$  If $\alpha=\{1,2, \dots, k+1\},$ then it is done. Assume $D_k[x_n^{(k+2)},(x_n^{(\beta_i)})_{i=1}^{k}; (f_j)_{j=1}^{k}] \to 0$ for some $\beta \in \mathcal{S}_{k+1}(k).$ Then, using \Cref{lem samir}, we have $D_k\left[\frac{1}{k+1}\sum_{i=1}^{k+1}x_n^{(i)},(x_n^{(\beta_i)})_{i=1}^{k}; (f_j)_{j=1}^{k}\right] \to 0.$ Thus, $D_k[(x_n^{(i)})_{i=1}^{k+1}; (f_j)_{j=1}^{k}] \to 0.$ Hence, $X$ is $k-$WUR.
\end{proof}
The proof of the subsequent corollary follows in  similar lines of the proof of \Cref{thrm WUR MLUR}.
\begin{corollary}
	Let $x \in S_X.$ Then the following statements are equivalent.
	\begin{enumerate}
		\item $X$ is $k-$WLUR at $x.$
		\item If $(x_n^{(1)}), (x_n^{(2)}), \dots, (x_n^{(k+2)})$ are $(k+2)-$sequences in $S_X$ such that $x_n^{(1)}=x$ for all $n \in \mathbb{N}$ and $\|(k+1)x_n^{(k+2)}-\sum_{i=1}^{k+1}x_n^{(i)}\| \to 0,$ then $D_k[(x_n^{(\alpha_i)})_{i=1}^{k+1}; (f_j)_{j=1}^{k}] \to 0$ for all $f_1, f_2, \dots, f_k \in S_{X^*}$ and $\alpha \in \mathcal{S}_{k+2}(k+1).$
		\item If $(x_n^{(1)}), (x_n^{(2)}), \dots, (x_n^{(k+2)})$ are $(k+2)-$sequences in $S_X$ such that $x_n^{(1)}=x$ for all $n \in \mathbb{N}$ and $\|(k+1)x_n^{(k+2)}-\sum_{i=1}^{k+1}x_n^{(i)}\| \to 0,$ then $D_k[(x_n^{(\alpha_i)})_{i=1}^{k+1}; (f_j)_{j=1}^{k}] \to 0$ for all $f_1, f_2, \dots, f_k \in S_{X^*}$  and for some  $\alpha \in \mathcal{S}_{k+2}(k+1).$
	\end{enumerate}
\end{corollary}
In the following proposition and example, we discuss some relationships between rotundity properties of a space and its double dual.

\begin{proposition}\label{prop WUR Rotund}
	If $X$ is $k-$WUR, then $X^{**}$ is $k-$rotund.
\end{proposition}
\begin{proof}
	Suppose $X^{**}$ is not $k-$rotund. Then there exist $(k+1)$ elements  $x_1^{**}, x_2^{**}, \dots, x_{k+1}^{**}$ in $S_{X^{**}}$ such that $\|\sum_{i=1}^{k+1}x_i^{**}\|=k+1,$ but $D_k[(x_i^{**})_{i=1}^{k+1}; (\tilde{g}_j)_{j=1}^{k}]=\epsilon$ for some $\tilde{g}_1, \tilde{g}_2, \dots, \tilde{g}_k \in S_{X^{***}}$ and $\epsilon>0.$ For every $1 \leq i \leq k+1,$ by Goldstine's theorem, there exists a net $(x_{\alpha_i}^{(i)})_{\alpha_i \in I_i}$ in $B_X$ such that $x_{\alpha_i}^{(i)} \xrightarrow{w^*} x_i^{**}.$ Then, by \cite[Page 150]{Megg1998}, there exists a subnet $(x_{\beta}^{(i)})$ of $(x_{\alpha_i}^{(i)})$ with the same index set for every $1 \leq i \leq k+1$. Now, using the  $w^*-$lower semi-continuity of the norm function, we have
	\[
	k+1 = \left\Vert \sum\limits_{i=1}^{k+1}x_i^{**} \right\Vert \leq \liminf_{\beta} \left\Vert \sum\limits_{i=1}^{k+1}x_{\beta}^{(i)} \right\Vert \leq \limsup_{\beta} \left\Vert \sum\limits_{i=1}^{k+1}x_{\beta}^{(i)} \right\Vert \leq k+1,
	\]
	which  implies $\|\sum_{i=1}^{k+1} x_{\beta}^{(i)}\| \to k+1.$ Therefore, by assumption, $D_k[(x_{\beta}^{(i)})_{i=1}^{k+1}; (f_j)_{j=1}^{k}] \to 0$ for all $f_1, f_2, \dots, f_k \in S_{X^*}.$ Since $x_{\beta}^{(i)} \xrightarrow{w^*} x_i^{**}$ for all $1 \leq i \leq k+1$, we have  $D_k[(x_{\beta}^{(i)})_{i=1}^{k+1}; (f_j)_{j=1}^{k}] \to D_k[(x_{i}^{**})_{i=1}^{k+1}; (f_j)_{j=1}^{k}]$ for all $f_1, f_2, \dots, f_k \in S_{X^*}.$
	Therefore, $D_k[(x_{i}^{**})_{i=1}^{k+1}; (f_j)_{j=1}^{k}]=0$ for all $f_1, f_2, \dots, f_{k} \in S_{X^*}.$ Further for every $1 \leq j \leq k,$ by Goldstine's theorem, there exists a net  $(f_{\lambda_j}^{(j)})_{\lambda_j \in J_j}$ in $B_{X^*}$ such that $f_{\lambda_j}^{(j)} \xrightarrow{w^*} \tilde{g}_j.$ Then for every $1 \leq j \leq k$ it is easy to find a subnet $(f_{\gamma}^{(j)})$ of $(f_{\lambda_j}^{(j)})$ with the same index set. Since $f_{\gamma}^{(j)} \xrightarrow{w^*} \tilde{g}_j$ for all $1 \leq j \leq k$, it follows that
	$D_k[(x_i^{**})_{i=1}^{k+1}; (f_{\gamma}^{(j)})_{j=1}^{k}] \to D_k[(x_i^{**})_{i=1}^{k+1}; (\tilde{g}_j)_{j=1}^{k}].$ Thus $D_k[(x_i^{**})_{i=1}^{k+1}; (\tilde{g}_j)_{j=1}^{k}] = 0,$ which is a contradiction. Hence $X^{**}$ is $k-$rotund.
\end{proof}
The following example illustrates that in \Cref{prop WUR Rotund} the assumption $k-$WUR cannot be replaced by $k-$LUR (hence, $k-$WLUR). Further, we will see in \Cref{eg WUR rotund suff} that the property $k-$WUR of a space $X$ is not sufficient for the space $X^{**}$ to be $k-$WMLUR. The converse of \Cref{prop WUR Rotund} need not be true in general. To see this, consider a strongly rotund space which is not $k-$WLUR (see, \Cref{eg kWUR,eg k+1 WUR}).
\begin{example}
	Let $X= (l_1, \|\cdot\|_1)$ and $k \in \mathbb{Z}^{+}.$ By \cite[Chapter II, Theorem 2.6]{DeGZ1993}, $X$ admits an equivalent norm (say, $\|\cdot\|_r)$ such that $Y= (l_1, \|\cdot\|_r)$ is LUR. Note that, by \cite[Chapter II, Corollary 3.5]{DeGZ1993}, $Y^*$ is not smooth. Thus, $Y^{**}$ is not rotund. Now, consider the Banach space $Z= l_2(Y).$ Then, by \cite[Theorem 1.1]{Lova1955}, $Z$ is LUR (hence, $k-$LUR). Clearly, $Z^{**} \cong l_2(Y^{**}).$ Therefore, by \cite[Corollary 2.10]{Veen2021}, $Z^{**}$ is not $k-$rotund.
\end{example}

We present some necessary and/or sufficient conditions for the notions $k-$WUR, $k-$WLUR and $k-$WMLUR  in terms of property $k-$$w$UC, $k-$$w$USCh and $k-$$w$SCh.

In the next result, we obtain some characterization of $k-$WUR in terms of property $k-$$w$UC.

\begin{theorem}\label{thrm WUR uc}
		Let  $r > 1.$ Then the following statements are equivalent.
	\begin{enumerate}
		\item $X$ is $k-$WUR.
		\item If $A$ and $B$ are non-empty subsets of $X$ such that $A$ is convex, then $(A,B)$ has property $k-$$w$UC.
		\item $(B_X,rS_X)$ has property $k-$$w$UC.
		\item $(S_X,rS_X)$ has property $k-$$w$UC.
			\end{enumerate}
\end{theorem}
\begin{proof} $(1) \Rightarrow (2)$: Let $A$ and $B$ be  non-empty  subsets   of $X$ and $A$ be convex. Let $(x_n^{(1)}), (x_n^{(2)}),$ $ \dots, (x_n^{(k+1)})$ be $(k+1)-$sequences in  $A,$ $(y_n)$ be a sequence in $B$ such that $\|x_n^{(i)}-y_n\| \to d(A,B)$ for all $1 \leq i \leq k+1$ and $f_1, f_2, \dots, f_k \in S_{X^*}.$ If $d(A,B)=0,$ then it is clear that $(A,B)$ has property $k-$$w$UC. Assume $d(A,B)>0.$ Since $A$ is convex, we have
	\[
	d(A,B) \leq \left \Vert  \frac{1}{k+1}\sum\limits_{i=1}^{k+1}x_n^{(i)}-y_n \right \Vert= \frac{1}{k+1}\left \Vert \sum\limits_{i=1}^{k+1} (x_n^{(i)}-y_n) \right \Vert \leq  \frac{1}{k+1} \sum\limits_{i=1}^{k+1}\| x_n^{(i)}-y_n \|
	\]
		and hence $\frac{1}{k+1}\| \sum_{i=1}^{k+1} (x_n^{(i)}-y_n) \| \to d(A,B).$ Now, by $(1)$, we have $D_k\left[\left(\frac{x_n^{(i)}-y_n}{d(A,B)}\right)_{i=1}^{k+1};(f_j)_{j=1}^{k}\right] \to 0.$ Therefore, by \Cref{remark vol}, we have
	$D_k[(x_n^{(i)})_{i=1}^{k+1};(f_j)_{j=1}^{k}] \to 0.$ Hence, $(A,B)$ has property $k-$$w$UC.\\
	$(2) \Rightarrow (3)$: Obvious.\\
	$(3) \Rightarrow (4)$: Since $S_X \subseteq B_X$ and $d(S_X, rS_X)= d(B_X, rS_X),$ it follows from the assumption  that $(S_X,rS_X)$ has property $k-$$w$UC.\\
	$(4) \Rightarrow (1)$: Let $(S_X, rS_X)$ has property $k-$$w$UC. By \Cref{kwusch kwuc} and  \Cref{thrm wUSCh}, it follows that $(S_X, (k+1)S_X)$ has property $k-$$w$UC. Let $(x_n^{(1)}), (x_n^{(2)}), \dots, (x_n^{(k+1)})$ be $(k+1)-$sequences in $S_{X}$ with $\|\sum_{i=1}^{k+1}x_n^{(i)}\| \to k+1$ and  $f_1,f_2,\dots, f_k \in S_{X^*}.$ For every $n \in \mathbb{N},$ define $y_n= \sum_{i=1}^{k+1}x_n^{(i)}.$ Then for all $1 \leq i \leq k+1,$ we have
\[
\left \Vert (k+1) \frac{y_n}{\|y_n\|}- x_n^{(i)} \right \Vert  = \left \Vert \frac{(k+1)x_n^{(i)}}{\|y_n\|}+ (k+1) \frac{y_n-x_n^{(i)}}{\|y_n\|}-x_n^{(i)} \right \Vert \\
 \leq \left \vert  \frac{k+1}{\|y_n\|}-1 \right \vert + \frac{(k+1)k}{\|y_n\|}.
\]	 
	Thus,  $\left \Vert x_n^{(i)}-(k+1) \frac{y_n}{\|y_n\|}\right \Vert \to d(S_X, (k+1)S_X)$ for all $1 \leq i \leq k+1.$ Since $(S_X,(k+1)S_X)$ has property $k-$$w$UC, we get $D_k[(x_n^{(i)})_{i=1}^{k+1};(f_j)_{j=1}^{k}] \to 0.$ Hence, $X$ is $k-$WUR.
\end{proof}
 The following corollary is an immediate consequence of \Cref{thrm wUSCh} and  \Cref{thrm WUR uc}. However, the converse need not be true in general.
\begin{corollary}\label{thrm WUR UC SX}
	Let $r \in (0,1).$ If $(S_X,rS_X)$ has property $k-$$w$UC, then $X$ is $k-$WUR.
\end{corollary}

Now, in view of \Cref{kwusch kwuc} and \Cref{thrm WUR uc}, we characterize  $k-$WUR spaces in terms of $k-$weakly uniformly strong  Chebyshevness of the corresponding closed unit ball.
\begin{theorem}\label{thrm k-WUR B_X}
	Let $r > 1.$ Then the following statements are equivalent.
	\begin{enumerate}
		\item  $X$ is $k-$WUR.
		\item  $B_X$ is $k-$$w$USCh on $X.$
		\item  $B_X$ is $k-$$w$USCh on $rS_X.$
	\end{enumerate}
\end{theorem}
\begin{proof}$(1) \Rightarrow (2)$: It is enough to show that $B_X$ is $k-$$w$USCh on $X \backslash B_X.$ Let  $(x_n^{(1)}), (x_n^{(2)}), \dots, $ $  (x_n^{(k+1)})$ be $(k+1)-$sequences in  $B_X$, $(y_n)$ be a sequence in  $X \backslash B_X$ with $\|x_n^{(i)}-y_n\|-d(y_n,B_X) \to 0$ for all $1 \leq i \leq k+1.$ Note that for all $n \in \mathbb{N},$ we have $d(y_n,B_X)=\|y_n\|-1,$ which implies $\|y_n\|-\|x_n^{(i)}-y_n\| \to 1$ for all $1 \leq i \leq k+1.$ Since
	\[
 \frac{1}{k+1}\left \Vert \sum\limits_{i=1}^{k+1}x_n^{(i)}\right \Vert \geq \|y_n\|- \frac{1}{k+1}\left \Vert\sum\limits_{i=1}^{k+1}x_n^{(i)}-(k+1)y_n \right \Vert \geq  \|y_n\|- \frac{1}{k+1}\sum\limits_{i=1}^{k+1}\|x_n^{(i)}-y_n\|,	
	\]
	it follows that  $\frac{1}{k+1}\| \sum_{i=1}^{k+1}x_n^{(i)} \| \to 1.$ Thus, by $(1),$ we have $D_k[(x_n^{(i)})_{i=1}^{k+1};(f_j)_{j=1}^{k}] \to 0$ for all  $f_1,f_2,\dots,f_k \in S_{X^*}.$ Therefore, $B_X$ is $k-$$w$USCh on $X \backslash B_X.$\\
	$(2) \Rightarrow (3)$: Obvious.\\
	$(3) \Rightarrow (1)$: By $(3)$ and \Cref{kwusch kwuc}, we have  $(B_X,rS_X)$ has property $k-$$w$UC. Thus, by  \Cref{thrm WUR uc}, it follows that $X$ is $k-$WUR.
\end{proof}

In light of \Cref{thrm WUR uc,thrm k-WUR B_X}, we now present few examples to illustrate some of the implications mentioned in Section $2$ cannot be reversed in general. As mentioned immediately after \Cref{def k-SCh,def property kwuc}, the following example shows that, in general $k-$$w$USCh (respectively, property $k-$$w$UC) does not imply $k-$SCh (respectively, property $k-$UC).
\begin{example}\label{eg converse Chev}
 Let $k \in \mathbb{Z}^+$. Consider the space  $X$  as in \Cref{eg kWUR not KMLUR}. Since $X$ is $k-$WUR, by \Cref{thrm k-WUR B_X}, $B_X$ is $k-$$w$USCh on $X.$ However as mentioned in \Cref{eg kWUR not KMLUR}, $B_X$ is not $k-$SCh on $X.$	In addition, observe that $X$ is $k-$WUR, but not $k-$UR. Hence, by \Cref{thrm WUR uc}, $(B_X, (k+1)S_X)$ has property $k-$$w$UC. However, by \cite[Theorem 2.19]{KaVe2018}, $(B_X, (k+1)S_X)$ does not have property $k-$UC.
\end{example}
As noted in Section $2,$ from the following example we can observe that the converses of the statements of \Cref{prop wUSCh (k+1)} are not necessarily true.

\begin{example}\label{eg k+1 not k USCh}
	Let $k \geq 2$. Consider a $k-$WUR space $X$ which is not $(k-1)-$rotund   (see, \Cref{eg k not k-1 WUR}). Therefore, by \Cref{thrm k-WUR B_X} and \cite[Proposition 2.4]{GaTh2023}, $B_{X}$ is $k-$$w$USCh on $2S_{X}$, but $B_X$ is not $(k-1)-$Chebyshev on $2S_{X}.$ Further, by \Cref{thrm WUR uc}, $(B_{X}, 2S_{X})$ has property $k-$$w$UC, but it  does not have property $(k-1)-$$w$UC.
\end{example}

For any non-empty closed convex subset $C$ of $X$ and $\alpha>0,$ we define $C^{\alpha}=\{x \in X: d(x, C)= \alpha\}.$ For any $x^* \in S_{X^*},$ we say that the set $ker(x^*)=\{x \in X: x^*(x)=0\}$ is a hyperplane of $X.$

 It follows from \Cref{kwusch kwuc} and \Cref{thrm WUR uc} that, every proximinal convex subset $C$ of a $k-$WUR space is $k-$$w$USCh on $C^{\alpha}$ for any $\alpha>0.$ In fact something more is true. To see this we define a notion called $k-$equi weakly uniform strong Chebyshevity as follows. Let $\mathcal{M}$ be a collection of proximinal convex subsets of $X$ and $\alpha>0.$ We say that $\mathcal{M}$ is $k-$equi weakly uniformly strongly Chebyshev (in short, $k-$equi $w$USCh) on $\mathcal{M}^{\alpha},$ if  for every $\epsilon>0$ and $f_1,f_2,\dots,f_k \in S_{X^*}$ there exists $\delta>0$ such that $\vert D_k[(x_i)_{i=1}^{k+1};(f_j)_{j=1}^{k}]\vert \leq \epsilon$ whenever $M \in \mathcal{M},$ $x \in M^{\alpha}$ and $x_1,x_2,\dots,x_{k+1} \in P_M(x,\delta).$

\begin{theorem}\label{thrm kWUR equi}
	Let $\alpha>0.$ Then the following statements are equivalent.
	\begin{enumerate}
		\item $X$ is $k-$WUR.
		\item $\mathcal{C}$ is $k-$equi $w$USCh on $\mathcal{C}^{\alpha},$ where $\mathcal{C}$ is the collection of all proximinal convex subsets of $X.$
		\item $\mathcal{M}$ is $k-$equi $w$USCh on $\mathcal{M}^{\alpha},$ where $\mathcal{M}$ is the collection of all proximinal subspaces of $X.$
		\item $\mathcal{H}$ is $k-$equi $w$USCh on $\mathcal{H}^{\alpha},$ where $\mathcal{H}$ is the collection of all proximinal hyperplanes of $X.$
		\item $\mathcal{F}$ is $k-$equi $w$USCh on $\mathcal{F}^{\alpha},$ where $\mathcal{F}$ is the collection of all $k-$dimensional subspaces of $X.$
	\end{enumerate}
\end{theorem}
	\begin{proof}
	$(1) \Rightarrow (2)$: Let $(C_n)$ be a sequence of proximinal convex subsets of $X.$ Let $(y_n^{(1)}), (y_n^{(2)}), \dots,$ $(y_n^{(k+1)})$ be $(k+1)-$sequences with $y_n^{(i)} \in C_n$ for all $n \in \mathbb{N}$ and $1 \leq i \leq k+1,$ $(x_n)$ be a sequence with  $x_n \in C_n^{\alpha}$ for all $n \in \mathbb{N}$ such that $\|y_n^{(i)}-x_n\| \to \alpha$ for all $1 \leq i \leq k+1$ and  $f_1, f_2, \dots, f_k \in S_{X^*}.$ Since $C_n$ is convex, we have
\[
 d(x_n,C_n) \leq \left\Vert \frac{1}{k+1}\sum\limits_{i=1}^{k+1}y_n^{(i)}-x_n \right\Vert= \frac{1}{k+1}\left\Vert\sum_{i=1}^{k+1}(y_n^{(i)}-x_n) \right\Vert \leq \frac{1}{k+1} \sum\limits_{i=1}^{k+1}\|y_n^{(i)}-x_n\|
\]
and hence $  \frac{1}{k+1}\Vert\sum_{i=1}^{k+1}(y_n^{(i)}-x_n) \Vert \to \alpha.$ Thus, by $(1),$ it follows that $D_k[(\frac{y_n^{(i)}-x_n}{\alpha})_{i=1}^{k+1}; (f_j)_{j=1}^{k}] \to 0.$ Therefore, by \Cref{volume property}, $D_k[(y_n^{(i)})_{i=1}^{k+1}; (f_j)_{j=1}^{k}] \to 0.$ \\
$(2) \Rightarrow (3) \Rightarrow (4)$: Obvious.\\
$(4) \Rightarrow (1)$: Let $(x_n^{(1)}), (x_n^{(2)}), \dots, (x_n^{(k+1)})$ be $(k+1)-$sequences in $S_X$ such that $\|\sum_{i=1}^{k+1}x_n^{(i)}\| \to k+1$ and $f_1, f_2, \dots, f_{k}\in S_{X^*}.$ For every $n \in \mathbb{N},$ define  $y_n= \frac{1}{k+1}\sum_{i=1}^{k+1}x_n^{(i)}.$ By Hahn-Banach theorem, for every $n \in \mathbb{N}$ there exists $g_n \in S_{X^*}$ such that $g_n(y_n)= \|y_n\|.$ Let $1 \leq i \leq k+1$. Observe that $g_n(y_n) \to 1$ and $g_n(x_n^{(i)}) \to 1$. Now, define $H_n= ker(g_n),$ $\beta_n= d(y_n,H_n)$ and $z_n^{(i)}= y_n-x_n^{(i)}-g_n(y_n-x_n^{(i)})\frac{y_n}{\|y_n\|}$ for all $n \in \mathbb{N}.$ Clearly $H_n$ is proximinal on $X$ for all $n \in \mathbb{N}$ and  $\beta_n \to 1.$ Note that $z_n^{(i)} \in H_n$ and $\frac{\alpha y_n}{\beta_n} \in H_n^{\alpha}$ for all  $n \in \mathbb{N}.$ Since
\[
	d(H_n, H_n^{\alpha}) \leq \left\Vert \frac{\alpha z_n^{(i)}}{\beta_n}- \frac{\alpha y_n}{\beta_n} \right\Vert\\
	 = \frac{\alpha}{\beta_n} \left\Vert x_n^{(i)}+ g_n(y_n-x_n^{(i)})\frac{y_n}{\|y_n\|}\right\Vert\\
	 \leq \frac{\alpha}{\beta_n} \left(\|x_n^{(i)}\|+ |g_n(y_n-x_n^{(i)})| \right),
\]
 we have $\| \frac{\alpha z_n^{(i)}}{\beta_n}- \frac{\alpha y_n}{\beta_n} \| \to \alpha.$ Thus, by $(4),$ $D_k[(\frac{\alpha}{\beta_n}z_n^{(i)})_{i=1}^{k+1}; (f_j)_{j=1}^{k}] \to 0.$ Further, using \Cref{volume property} and  \Cref{lem samir}, we have $D_k[(x_n^{(i)})_{i=1}^{k+1}; (f_j)_{j=1}^{k}] \to 0.$ Hence, $X$ is $k-$WUR. \\
$(2) \Rightarrow (5)$: Obvious.\\
$(5) \Rightarrow (1)$: Suppose there exist $\epsilon>0,$ $g_1, g_2, \dots, g_{k} \in S_{X^*}$ and  $(k+1)-$sequences  $(x_n^{(1)}), (x_n^{(2)}),$ $\dots,(x_n^{(k+1)})$  in $S_X$ such that $\|\sum_{i=1}^{k+1}x_n^{(i)}\| \to k+1,$ but $|D_k[(x_n^{(i)})_{i=1}^{k+1}; (g_j)_{j=1}^{k}]|> \epsilon$ for all $n \in \mathbb{N}.$ Now for every $n \in \mathbb{N},$ define $F_n= span\{x_n^{(i)}- x_n^{(k+1)}: 1 \leq i \leq k\}.$ Using  \Cref{volume property}, observe that $F_n$ is a $k-$dimensional subspace of $X$ and  hence it is proximinal on $X.$ Thus, for every $n \in \mathbb{N},$ there exist  $\lambda_n^{(1)}, \lambda_n^{(2)}, \dots, \lambda_n^{(k)} \in \mathbb{R}$ such that  $\|x_n^{(k+1)}+\sum_{i=1}^{k}\lambda_n^{(i)}(x_n^{(i)}- x_n^{(k+1)})\| = d(x_n^{(k+1)}, F_n).$ Denote $d(x_n^{(k+1)}, F_n)= \beta_n$ for all $n \in \mathbb{N}.$ Using \cite[Lemma 2.3]{VeVe2018}, we have $\beta_n \to 1.$ Note that  $d(\frac{\alpha}{\beta_n}x_n^{(k+1)}, F_n)=\alpha,$ $\frac{\alpha}{\beta_n}(x_n^{(k+1)}- x_n^{(i)}) \in F_n$ and $\|\frac{\alpha}{\beta_n}(x_n^{(k+1)}-x_n^{(i)})- \frac{\alpha}{\beta_n}x_n^{(k+1)}\| \to \alpha$ for all $1 \leq i \leq k+1.$ Therefore, our assumption leads to $D_k[(\frac{\alpha}{\beta_n}(x_n^{(k+1)}-x_n^{(i)}))_{i=1}^{k+1}; (g_j)_{j=1}^{k}] \to 0$. Thus from \Cref{remark vol} and \Cref{lem samir},  $D_k[(x_n^{(i)})_{i=1}^{k+1}; (g_j)_{j=1}^{k}] \to 0,$  which is a contradiction. Hence the proof.
\end{proof}
We remark that \Cref{thrm WUR uc,thrm k-WUR B_X,thrm kWUR equi} are generalizations of \cite[Theorems 4.5, 4.6 and 4.15]{GaTh2022}.

		In the next two results, we present a necessary and a sufficient condition for a space to be $k-$WLUR in terms of $k-$weakly strongly Chebyshevness.

\begin{proposition}\label{prop kWLUR prox convex}
		If $X$ is a $k-$WLUR space, then every proximinal convex subset of $X$ is $k-$$w$SCh on $X.$
	\end{proposition}
	\begin{proof}
		In view of \Cref{remark vol}, it is enough to prove that every proximinal convex subset  $C$ of $X$ with $d(0, C)=1$ is $k-$$w$SCh at $0.$ Let $(x_n^{(1)}), (x_n^{(2)}),\ldots,(x_n^{(k+1)})$ be $(k+1)-$sequences in $C$ such that $\|x_n^{(i)}\| \xrightarrow{} 1$ for all $1\leq i \leq k+1$ and $f_1, f_2, \dots, f_k \in S_{X^*}.$ Choose $y\in P_C(0)$ and observe that $y\in S_X.$ Since $C$ is convex,   for any $\alpha \in \mathcal{S}_{k+1}(k)$, we have
		\[ 1= d(0, C) \leq \frac{1}{k+1}\left\|y+\sum_{i=1}^k x_n^{(\alpha_i)}\right\| \leq \frac{1}{k+1}(\|y\|+\sum_{i=1}^k \|x_n^{(\alpha_i)}\|)\]
		and hence $\|y+\sum_{i=1}^k x_n^{(\alpha_i)}\| \xrightarrow{} k+1.$ Thus, by assumption, we have $D_k[y, (x_n^{(\alpha_i)})_{i=1}^k; (f_j)_{j=1}^k] \xrightarrow{} 0$ for any $\alpha \in \mathcal{S}_{k+1}(k).$ Since, by \cite[Lemma 2]{Suya2000},
		\[ |D_k[(x_n^{(i)})_{i=1}^{k+1}; (f_j)_{j=1}^k]| \leq \sum_{\alpha \in \mathcal{S}_{k+1}(k) } |D_k[y, (x_n^{(\alpha_i)})_{i=1}^k; (f_j)_{j=1}^k]|,\]
	    we have $D_k[(x_n^{(i)})_{i=1}^{k+1}; (f_j)_{j=1}^k] \xrightarrow{} 0.$
	\end{proof}
We remark that the converse of \Cref{prop kWLUR prox convex} not necessarily true (see, \Cref{eg converse Chev2}).
		\begin{theorem}\label{coro WLUR S_X}
			Consider the following statements.
			\begin{enumerate}
				\item $X$ is $k-$WLUR.
				\item $S_X$ is $k-$$w$SCh on $rS_X$ for some $r \in (0,1).$
				\item $(S_X, C)$ has property $k-$$w$UC, whenever $C$ is a non-empty boundedly compact subset of $X$ with $d(0,C)>0.$
			\end{enumerate}
			Then $(1) \Leftarrow (2) \Leftrightarrow (3).$
		\end{theorem}		
		\begin{proof}
			$(2) \Rightarrow (1)$: 	From the assumption and  \Cref{thrm wUSCh}, we have $S_X$ is $k-$$w$SCh on $\frac{1}{2}S_X.$ Let $x \in S_X,$ $(x_n^{(1)}),(x_n^{(2)}), \dots, (x_n^{(k)})$ be $(k)-$sequences in $S_X$ such that $\Vert x+\sum_{i=1}^{k}x_n^{(i)} \Vert \to k+1$ and $f_1,f_2,\dots,f_k \in S_{X^*}.$  Observe that for any $1 \leq i \leq k,$ we have $\|x_n^{(i)}+x\| \to 2$ and
			\[
			d(\frac{1}{2}x, S_X) \leq \left \Vert \frac{x_n^{(i)}+x}{{\|x_n^{(i)}+x}\|}-\frac{1}{2}x \right \Vert \leq \left \vert \frac{1}{\|x_n^{(i)}+x\|}-\frac{1}{2} \right \vert + \frac{1}{\|x_n^{(i)}+x\|},
			\]
			which implies $\left \Vert \frac{x_n^{(i)}+x}{\|x_n^{(i)}+x\|}-\frac{1}{2}x \right \Vert \to d(\frac{1}{2}x, S_X).$ Thus, we have
			$
			D_k \left[x, \left(\frac{x_n^{(i)}+x}{\|x_n^{(i)}+x\|}\right)_{i=1}^{k};(f_j)_{j=1}^{k}\right] \to 0.
			$
			Therefore, by \Cref{remark vol} and  \Cref{lem samir}, it follows that	$D_k [x, (x_n^{(i)})_{i=1}^{k};(f_j)_{j=1}^{k}] \to 0.$ \\
			$(2) \Rightarrow (3)$: This implication follows from \Cref{rem boundely compact}.\\
			$(3) \Rightarrow (2)$: This implication follows from the \Cref{def k-SCh,def property kwuc}.
		\end{proof}
	We note that for the case $k=1$,  \Cref{prop kWLUR prox convex} and $(2) \Rightarrow (1)$ of \Cref{coro WLUR S_X} are proved in  \cite[Propositon 4.7]{GaTh2022} and \cite[Theorem 3.12]{DuSh2018} respectively.
	In the following example, we  see that the 	other implications of   \Cref{thrm WUR UC SX} and \Cref{coro WLUR S_X} need not be true in general for any $k \geq 2$. For instance for any $k \geq 2$, consider the $k-$WUR space $X=(\mathbb{R}^{k}, \|\cdot\|_{\infty})$ and $x=(r,r, \dots, r)$  for some $0 < r<1.$ It is easy to see that, $S_X$ is not $k-$Chebyshev at $x,$ hence $S_X$ is not $k-$$w$SCh on $rS_X.$ Further, by \Cref{kwusch kwuc}, $(S_X, rS_X)$ does not have property $k-$$w$UC.

	 In light of the \Cref{def KWLUR} and  \Cref{thrm WUR MLUR}, it is natural to ask whether the local version of \Cref{thrm k-WUR B_X} holds. The following result provides a positive answer to this question. We conclude this section with some characterizations of the $k-$WMLUR spaces.
	\begin{theorem}\label{thrm MLUR B_X}
		Let $r>1.$ Then the following statements are equivalent.
		\begin{enumerate}
			\item $X$ is $k-$WMLUR.
			\item $B_X$ is $k-$$w$SCh on $X.$
			\item $B_X$ is $k-$$w$SCh on $rS_X.$
		\end{enumerate}
	\end{theorem}
		\begin{proof}
			$(1) \Rightarrow (2)$: In view of the \Cref{thrm wUSCh}, it is enough to prove that $B_X$ is $k-$$w$SCh on $\frac{k+1}{k}S_X.$ Let $x \in S_X,$  $(x_n^{(1)}), (x_n^{(2)}), \dots, (x_n^{(k+1)})$ be $(k+1)-$sequences in $B_X$ such that $\|x_n^{(i)}- \frac{k+1}{k}x\| \to \frac{1}{k}$ and  $f_1, f_2, \dots, f_k \in S_{X^*}.$ We need to  show that $D_k[(x_n^{(i)})_{i=1}^{k+1}; (f_j)_{j=1}^{k}] \to 0.$ Observe that for any $1 \leq i \leq k+1,$
			\[
			\frac{k+1}{k}-1 \leq  \frac{k+1}{k}\|x\| - \|x_n^{(i)}\| \leq \left \Vert x_n^{(i)}- \frac{k+1}{k}x\right \Vert
			\]
			holds, which implies $\|x_n^{(i)}\| \to 1.$ Now, we claim that  $D_{k}[x,(x_n^{(\alpha_i)})_{i=1}^{k}; (f_j)_{j=1}^{k}] \to 0$ for all $\alpha \in \mathcal{S}_{k+1}(k).$ Let $\alpha \in \mathcal{S}_{k+1}(k).$ For every $n \in \mathbb{N},$ define $z_n=(k+1)x-\sum_{i=1}^{k}  x_n^{(\alpha_i)}.$ Since
			\[
			d\left(B_X, \frac{k+1}{k}S_X\right)\leq \left\Vert\frac{1}{k}\sum\limits_{i=1}^{k}x_n^{(\alpha_i)}- \frac{k+1}{k}x\right\Vert \leq \frac{1}{k}\sum\limits_{i=1}^{k}\left\Vert x_n^{(\alpha_i)}- \frac{k+1}{k}x \right\Vert,
				\]
				we have $\|z_n\| \to 1.$ Observe that $\|(k+1)x-(\sum_{i=1}^{k}x_n^{(\alpha_i)}+z_n)\|=0$ for all $n \in \mathbb{N}.$ Hence, by $(1),$ we get $D_k[z_n, (x_n^{(\alpha_i)})_{i=1}^{k}; (f_j)_{j=1}^{k}] \to 0.$  Note that $|D_k[x,(x_n^{(\alpha_i)})_{i=1}^{k}; (f_j)_{j=1}^{k}]| = \frac{1}{k+1} |D_k[z_n,(x_n^{(\alpha_i)})_{i=1}^{k}; (f_j)_{j=1}^{k}]|$ for all $n \in \mathbb{N}.$ Thus, $D_k[x,(x_n^{(\alpha_i)})_{i=1}^{k}; (f_j)_{j=1}^{k}] \to 0.$  Hence the claim.\\
				Now, it follows from \cite[Lemma 2]{Suya2000} that
		\[
		|D_{k}[(x_n^{(i)})_{i=1}^{k+1}; (f_j)_{j=1}^{k}]| \leq \sum\limits_{\alpha \in \mathcal{S}_{k+1}(k)}|D_{k}[x,(x_n^{(\alpha_i)})_{i=1}^{k}; (f_j)_{j=1}^{k}],
		\]
		which implies 	$D_{k}[(x_n^{(i)})_{i=1}^{k+1}; (f_j)_{j=1}^{k}] \to 0.$ Hence, $B_X$ is $k-$$w$SCh on $\frac{k+1}{k}S_X.$\\
		$(2) \Rightarrow (3)$: Obvious.\\
		$(3) \Rightarrow (1)$: It follows from the assumption  and \Cref{thrm wUSCh} that $B_X$ is $k-$$w$SCh on $(k+1)S_X.$ Let $x \in S_X,$  $(x_n^{(1)}), (x_n^{(2)}), \dots, (x_n^{(k+1)})$ be $(k+1)-$sequences in $S_X$ such that $\|(k+1)x-\sum_{i=1}^{k+1} x_n^{(i)}\| \to 0$ and $f_1, f_2, \dots, f_k \in S_{X^*}.$ We need to show that  $D_k[(x_n^{(i)})_{i=1}^{k+1}; (f_j)_{j=1}^{k}] \to 0.$ Note that for any $1 \leq i \leq k+1,$ we have
		\[
		\begin{aligned}
			(k+1)\|x\|- \|x_n^{(i)}\| & \leq \|x_n^{(i)}-(k+1)x\|\\
			&\leq \left\Vert \sum\limits_{j=1}^{k+1}x_n^{(j)}-(k+1)x \right\Vert + \left\Vert \sum\limits_{j=1, j \neq i}^{k+1}x_n^{(j)}\right\Vert \\
			&\leq \left\Vert \sum\limits_{j=1}^{k+1}x_n^{(j)}-(k+1)x \right\Vert+k.
		\end{aligned}
		\]	
	Thus, $\|x_n^{(i)}-(k+1)x\| \to k.$ Since $B_X$ is $k-$$w$SCh on $(k+1)S_X,$ we get $D_k[(x_n^{(i)})_{i=1}^{k+1}; (f_j)_{j=1}^{k}] \to 0.$ Hence, $X$ is $k-$WMLUR.
		\end{proof}
	We remark that for the case $k=1,$ \Cref{thrm MLUR B_X} is proved in \cite[Theorem 2.6]{ZhLZ2015}. The following corollary is an immediate consequence of \Cref{rem boundely compact} and \Cref{thrm MLUR B_X}.
	\begin{corollary}
		The following statements are equivalent.
		\begin{enumerate}
			\item $X$ is $k-$WMLUR.
			\item  If $A$ is a closed ball in $X$ and $B$ is a non-empty boundedly compact subset of $X,$ then $(A,B)$ has property $k-$$w$UC.
		\end{enumerate}
	\end{corollary}
As specified immediately after \Cref{def KWLUR}, the following example illustrates that, in general, $k-$rotundity does not imply $k-$WMLUR.
\begin{example}\label{eg kWMLUR not kR}
	Let $k \in \mathbb{Z}^{+}.$ Consider the space $X$ as in \Cref{eg Ch not wSCh}. Note that $X$ is $k-$rotund, but not MLUR. From \Cref{eg Ch not wSCh}, it is clear that $B_X$ is not $k-$$w$SCh on $X.$ Therefore, by \Cref{thrm MLUR B_X}, $X$ is not $k-$WMLUR.
	\end{example}
\section{Stability of $k-$WUR, $k-$WLUR and $k-$WMLUR }	
In this section, we examine the stability of the notions $k-$WUR, $k-$WLUR and $k-$WMLUR. We begin with the inheritance nature of the notions $k-$WUR and $k-$WLUR by quotient spaces.

In view of \Cref{thrm quotiet k-rotund}, it is natural to ask whether a similar characterization holds for the notions $k-$WUR and $k-$WLUR. To answer this question, in the following result, we prove that the collection of all quotient spaces of a $k-$WUR space is uniformly $k-$WUR. Indeed the reverse implication also holds.

	\begin{theorem}\label{thrm WUR quot}
			Let $\alpha, \beta \in \mathbb{Z}^{+}$, $X$ be a Banach space satisfying  $dim(X) \geq k+2,$ $1 \leq \alpha \leq dim(X)-(k+1)$ and  $k+1 \leq \beta \leq  dim(X)-1.$ Then the following statements are equivalent.
		\begin{enumerate}
			\item $X$ is $k-$WUR.
			\item For every $\epsilon>0$ and $f_1, f_2, \dots, f_{k} \in S_{X^*},$ it follows that \begin{center}$\inf \{\delta^{k}_{X/M}(\epsilon, (f_j)_{j=1}^{k}): M \subseteq \cap_{j=1}^{k}ker(f_j) \}>0.$\end{center}
			\item For every $\epsilon>0$ and $f_1, f_2, \dots, f_{k} \in S_{X^*},$ it follows that \begin{center}
				$\inf \{\delta^{k}_{X/F}(\epsilon, (f_j)_{j=1}^{k}): F \subseteq \cap_{j=1}^{k}ker(f_j)$ with  $dim(F)= \alpha \}>0.$\end{center}
			\item  For every $\epsilon>0$ and $f_1, f_2, \dots, f_{k} \in S_{X^*},$ it follows that \begin{center}$\inf \{\delta^{k}_{X/Y}(\epsilon, (f_j)_{j=1}^{k}): Y\subseteq \cap_{j=1}^{k}ker(f_j)$ with $codim(Y)= \beta \}>0.$\end{center}
	\end{enumerate}
	\end{theorem}
	\begin{proof}
		$(1) \Rightarrow (2)$: Suppose there exist $\epsilon>0,$ $f_1, f_2, \dots, f_{k} \in S_{X^*}$, a sequence $(M_n)$ of  subspaces in  $X$ such that  $M_n \subseteq \cap_{j=1}^{k}ker(f_j)$ and  $ \delta^{k}_{X/M_n}(\epsilon, (f_j)_{j=1}^{k})< \frac{1}{n}$ for all $n \in \mathbb{N}.$ Then there exist $(k+1)-$sequences  $(x_n^{(1)}+M_n), (x_n^{(2)}+M_n), \dots, (x_n^{(k+1)}+M_n)$ with  $x_n^{(i)}+M_n \in S_{X/M_n}$ for all $n \in \mathbb{N},$ $1 \leq i \leq k+1$ such that $|D_k[(x_n^{(i)}+M_n)_{i=1}^{k+1}; (f_j)_{j=1}^{k}]| \geq \epsilon$ for all $n \in \mathbb{N}$, but $\|\sum_{i=1}^{k+1}(x_n^{(i)}+M_n)\| \to k+1.$   Since $d(x_n^{(i)}, M_n)=1,$ there exists $y_n^{(i)} \in M_n$ such that $ \|x_n^{(i)}-y_n^{(i)}\| \to 1$ for all $1 \leq i \leq k+1.$ Therefore, we have
		\[
		\left\Vert\sum_{i=1}^{k+1}x_n^{(i)}+M_n\right\Vert \leq \left \Vert \sum\limits_{i=1}^{k+1}x_n^{(i)}- \sum\limits_{i=1}^{k+1}y_n^{(i)} \right\Vert \leq \sum\limits_{i=1}^{k+1}\|x_n^{(i)}-y_n^{(i)}\| 		
		\]
		and hence $\|\sum_{i=1}^{k+1}(x_n^{(i)}-y_n^{(i)})\| \to k+1.$ By $(1),$ we get  $D_{k}[(x_n^{(i)}-y_n^{(i)})_{i=1}^{k+1}; (f_j)_{j=1}^{k}] \to 0.$ Thus, by \Cref{remark vol}, we have $D_k[(x_n^{(i)}+M_n)_{i=1}^{k+1}; (f_j)_{j=1}^{k}] \to 0.$ This is a contradiction. \\
		$(2) \Rightarrow (3)$: Obvious. \\
		$(3) \Rightarrow (1)$: Suppose there exist $\epsilon>0,$ $f_1, f_2, \dots, f_{k} \in S_{X^*}$ and $(k+1)-$sequences $(x_n^{(1)}), (x_n^{(2)}), \dots,$ $(x_n^{(k+1)})$ in $S_X$ such that $\|\sum_{i=1}^{k+1}x_n^{(i)}\| \to k+1,$ but $|D_{k}[(x_n^{(i)})_{i=1}^{k+1}; (f_j)_{j=1}^{k}]| \geq \epsilon$ for all $n \in \mathbb{N}.$ By Hahn-Banach theorem, for every  $n \in \mathbb{N}$ there exists $g_n \in S_{X^*}$ such that $g_n(\sum_{i=1}^{k+1}x_n^{(i)})= \| \sum_{i=1}^{k+1}x_n^{(i)}\|.$  Now, for every $n \in \mathbb{N},$ choose a subspace $F_n$ of $X$ such that $F_n \subseteq (\cap_{j=1}^{k}ker(f_j)) \cap ker(g_n)$ and $dim(F_n)= \alpha.$ Let $1 \leq i \leq k+1.$ Since $g_n(x_n^{(i)}) \to 1$ and
		\[
		|g_n(x_n^{(i)})|= d(x_n^{(i)}, ker(g_n)) \leq d(x_n^{(i)}, F_n) \leq \|x_n^{(i)}\| =1,
		\]
		it follows that $\|x_n^{(i)}+F_n\| \to 1.$ For every $n \in \mathbb{N},$ define $y_n^{(i)}= \frac{x_n^{(i)}}{d(x_n^{(i)},F_n)}.$ Note that $y_n^{(i)}+F_n \in S_{X/F_n}$ for all $n \in \mathbb{N}$ and $g_n(y_n^{(i)}) \to 1.$ Therefore, we have
		\[
		\left \vert g_n \left( \sum\limits_{i=1}^{k+1} y_n^{(i)} \right) \right\vert= d\left(\sum\limits_{i=1}^{k+1} y_n^{(i)}, ker(g_n)\right) \leq  d\left(\sum\limits_{i=1}^{k+1} y_n^{(i)}, F_n\right) \leq  \sum\limits_{i=1}^{k+1}\left \Vert y_n^{(i)}+F_n \right \Vert=k+1
		\]
		and hence $\|\sum_{i=1}^{k+1} y_n^{(i)}+F_n\| \to k+1.$ Thus, by assumption, we get $D_k[(y_n^{(i)}+F_n)_{i=1}^{k+1}; (f_j)_{j=1}^{k}] \to 0.$ Further, by \Cref{remark vol} and \Cref{lem samir}, it follows that $D_k[(x_n^{(i)})_{i=1}^{k+1}; (f_j)_{j=1}^{k}] \to 0.$ This is a contradiction.\\
		$(2) \Rightarrow (4)$: Obvious.\\
		$(4) \Rightarrow (1)$: Suppose there exist $\epsilon>0,$ $f_1, f_2, \dots, f_{k} \in S_{X^*}$ and $(k+1)-$sequences $(x_n^{(1)}), (x_n^{(2)}), \dots,$ $(x_n^{(k+1)})$ in $S_X$ such that $\|\sum_{i=1}^{k+1}x_n^{(i)}\| \to k+1,$ but $D_{k}[(x_n^{(i)})_{i=1}^{k+1}; (f_j)_{j=1}^{k}] \geq \epsilon$ for all $n \in \mathbb{N}.$ By Hahn-Banach theorem, for every  $n \in \mathbb{N}$ there exists $g_n \in S_{X^*}$ such that $g_n(\sum_{i=1}^{k+1}x_n^{(i)})= \| \sum_{i=1}^{k+1}x_n^{(i)}\|.$  Now, for every $n \in \mathbb{N},$ choose a subspace $Y_n$ of $X$ such that $Y_n \subseteq (\cap_{j=1}^{k}ker(f_j)) \cap ker(g_n)$ and $codim(Y_n)= \beta.$ By replacing $F_n$ by $Y_n$ in the proof of $(3) \Rightarrow (1)$ and repeating the argument involved in the proof, we get a contradiction. Hence the proof.
		\end{proof}
	The following corollary is an immediate consequence of \Cref{thrm WUR quot}.
	\begin{corollary}\label{coro WUR quot}
		If $X$ is $k-$WUR and $M$ is a subspace of $X,$ then $X/M$ is $k-$WUR.
	\end{corollary}
 Now, we present an analogous result of \Cref{thrm WUR quot} for the notion $k-$WLUR.
	\begin{theorem}\label{thrm WLUR quot}
		Let $\alpha, \beta \in \mathbb{Z}^{+},$ $X$ be a Banach space satisfying $dim(X) \geq k+3,$ $1 \leq \alpha \leq dim(X)-(k+2)$, $k+2 \leq \beta \leq dim(X)-1$ and $x \in S_X.$ Then the following statements are equivalent.
		\begin{enumerate}
			\item $X$ is $k-$WLUR at $x.$
			\item  For every $\epsilon>0$ and $f_1, f_2, \dots, f_{k} \in S_{X^*},$ it follows that
			\begin{center} $\inf \{\delta^{k}_{X/M}(\epsilon, x+M, (f_j)_{j=1}^{k}):  M \subseteq \cap_{j=1}^{k}ker(f_j) \ \mbox{and} \ d(x,M)=1  \}>0.$
				\end{center}
			\item For every $\epsilon>0$ and $f_1, f_2, \dots, f_{k} \in S_{X^*},$ it follows that
			\begin{center}
			$\inf \{\delta^{k}_{X/F}(\epsilon, x+F, (f_j)_{j=1}^{k}): dim(F)= \alpha \ \mbox{with} \ F \subseteq \cap_{j=1}^{k}ker(f_j) \ \mbox{and} \ d(x,F)=1  \}>0.$
		\end{center}
			\item  For every $\epsilon>0$ and $f_1, f_2, \dots, f_{k} \in S_{X^*},$ it follows that \begin{center}$\inf \{\delta^{k}_{X/Y}(\epsilon, x+Y, (f_j)_{j=1}^{k}): codim(Y)= \beta \ \mbox{with} \ Y\subseteq \cap_{j=1}^{k}ker(f_j) \ \mbox{and} \ d(x,Y)=1 \}>0.$\end{center}
		\end{enumerate}
	\end{theorem}
	\begin{proof}
		$(1) \Rightarrow (2)$: Suppose there exist $\epsilon>0,$ $f_1, f_2, \dots, f_{k} \in S_{X^*},$ a sequence $(M_n)$ of   subspaces in  $X$ such that  $M_n\subseteq \cap_{j=1}^{k}ker(f_j)$, $d(x, M_n)=1$ and $ \delta^{k}_{X/M_n}(\epsilon, x+M_n, (f_j)_{j=1}^{k})< \frac{1}{n}$ for all $n \in \mathbb{N}.$ Then there exist $(k)-$sequences  $(x_n^{(1)}+M_n), (x_n^{(2)}+M_n), \dots, (x_n^{(k)}+M_n)$ with  $x_n^{(i)}+M_n \in S_{X/M_n}$ for all $n \in \mathbb{N}$ and $1 \leq i \leq k$ such that $|D_k[(x+M_n),(x_n^{(i)}+M_n)_{i=1}^{k}; (f_j)_{j=1}^{k}]| \geq \epsilon$ for all $n \in \mathbb{N}$, but $\|(x+M_n)+\sum_{i=1}^{k}(x_n^{(i)}+M_n)\| \to k+1$.   Since $d(x_n^{(i)}, M_n)=1,$ there exists $y_n^{(i)} \in M_n$ such that  $\|x_n^{(i)}-y_n^{(i)}\| \to 1$ for all $1 \leq i \leq k.$ Therefore, we have
		\[
		\left\Vert (x+M_n)+\sum_{i=1}^{k}(x_n^{(i)}+M_n)\right\Vert \leq \left \Vert x+\sum\limits_{i=1}^{k}x_n^{(i)}- \sum\limits_{i=1}^{k}y_n^{(i)} \right\Vert \leq \|x\|+ \sum\limits_{i=1}^{k}\|x_n^{(i)}-y_n^{(i)}\| 		
		\]
		and hence $\|x+\sum_{i=1}^{k}(x_n^{(i)}-y_n^{(i)})\| \to k+1.$ By $(1),$ we get  $D_{k}[x, (x_n^{(i)}-y_n^{(i)})_{i=1}^{k}; (f_j)_{j=1}^{k}] \to 0.$ Thus, by \Cref{remark vol}, we have $D_k[(x+M_n),(x_n^{(i)}+M_n)_{i=1}^{k}; (f_j)_{j=1}^{k}] \to 0.$ This is a contradiction. \\
		$(2) \Rightarrow (3)$: Obvious. \\
		$(3) \Rightarrow (1)$: Suppose there exist $\epsilon>0,$ $f_1, f_2, \dots, f_{k} \in S_{X^*}$ and $(k)-$sequences $(x_n^{(1)}), (x_n^{(2)}), \dots,$ $(x_n^{(k)})$ in $S_X$ such that $\|x+\sum_{i=1}^{k}x_n^{(i)}\| \to k+1,$ but $|D_{k}[x, (x_n^{(i)})_{i=1}^{k}; (f_j)_{j=1}^{k}]| \geq \epsilon$ for all $n \in \mathbb{N}.$ By Hahn-Banach theorem, there exists $g \in S_{X^*}$ such that $g(x)=\|x\|$ and for every  $n \in \mathbb{N},$ there exists $g_n \in S_{X^*}$ such that $g_n(x+\sum_{i=1}^{k+1}x_n^{(i)})= \|x+ \sum_{i=1}^{k+1}x_n^{(i)}\|.$ Observe that $g_n(x_n^{(i)}) \to 1$ for all $1 \leq i \leq k$ and $g_n(x) \to 1.$ Now, for every $n \in \mathbb{N},$ choose a subspace $F_n$  of $X$ such that $F_n \subseteq (\cap_{j=1}^{k}ker(f_j)) \cap ker(g_n) \cap ker(g)$ and $dim(F_n)= \alpha.$ Let  $1 \leq i \leq k.$ Since
		\[
		|g_n(x_n^{(i)})|= d(x_n^{(i)}, ker(g_n)) \leq d(x_n^{(i)}, F_n) \leq \|x_n^{(i)}\| =1,
		\]
	we have  $\|x_n^{(i)}+F_n\| \to 1.$ Similarly, $\|x+ F_n\|=1$ for all $n \in \mathbb{N}.$ For every $n \in \mathbb{N}$, define $y_n^{(i)}= \frac{x_n^{(i)}}{d(x_n^{(i)}, F_n)}$ and observe that $y_n^{(i)}+F_n \in S_{X/F_n}$ and $g_n(y_n^{(i)}) \to 1$. Therefore, we have
		\[
		\left \vert g_n \left(x+ \sum\limits_{i=1}^{k} y_n^{(i)} \right) \right\vert= d\left(x+\sum\limits_{i=1}^{k} y_n^{(i)}, ker(g_n)\right) \leq  d\left(x+\sum\limits_{i=1}^{k} y_n^{(i)}, F_n\right) \leq k+1
		\]
		and hence $\|(x+F_n)+\sum_{i=1}^{k} (y_n^{(i)}+F_n)\| \to k+1.$ By $(3),$ we get $D_k[(x+F_n), (y_n^{(i)}+F_n)_{i=1}^{k}; (f_j)_{j=1}^{k}] \to 0,$ Thus, by \Cref{remark vol} and  \Cref{lem samir}, $D_k[x, (x_n^{(i)})_{i=1}^{k}; (f_j)_{j=1}^{k}] \to 0.$ This is a contradiction.\\
		$(2) \Rightarrow (4)$: Obvious.\\
		$(4) \Rightarrow (1)$: Suppose there exist $\epsilon>0,$ $f_1, f_2, \dots, f_{k} \in S_{X^*}$ and $(k)-$sequences $(x_n^{(1)}), (x_n^{(2)}), \dots,$ $(x_n^{(k)})$ in $S_X$ such that $\|x+\sum_{i=1}^{k}x_n^{(i)}\| \to k+1,$ but $|D_{k}[x, (x_n^{(i)})_{i=1}^{k}; (f_j)_{j=1}^{k}]| \geq \epsilon$ for all $n \in \mathbb{N}.$ By Hahn-Banach theorem, there exists $g \in S_{X^*}$ such that $g(x)=\|x\|$ and for every  $n \in \mathbb{N},$ there exists $g_n \in S_{X^*}$ such that $g_n(x+\sum_{i=1}^{k+1}x_n^{(i)})= \|x+ \sum_{i=1}^{k+1}x_n^{(i)}\|.$ Now, for every $n \in \mathbb{N},$ choose a subspace $Y_n$ of $X$ such that $Y_n \subseteq (\cap_{j=1}^{k}ker(f_j)) \cap ker(g_n) \cap ker(g)$ and $codim(Y_n)= \beta.$  By replacing $F_n$ by $Y_n$ in the proof of $(3) \Rightarrow (1)$ and proceeding in a similar way, we get a contradiction. Hence the proof.
	\end{proof}
	As a consequence of \Cref{thrm WLUR quot}, we have the following result.
	\begin{corollary}\label{coro WLUR quot}
		If $X$ is $k-$WLUR and $M$ is a proximinal subspace of $X,$ then $X/M$ is $k-$WLUR.
	\end{corollary}
\begin{proof}
	 Let $\epsilon >0$, $x+M \in S_{X/M}$ and $f_1, f_2, \dots, f_k \in S_{M^{\perp}} \cong S_{(X/M)^*}$. By assumption, there exists $y \in M$ such that $d(x, M)=\|x-y\|=1$. Since  $X$ is $k-$WLUR at $(x-y)$,  it follows from \Cref{thrm WLUR quot} that $\delta_{X/M}^{k}(\epsilon, (x-y)+M, (f_j)_{j=1}^{k})>0$. Thus, $X/M$ is $k-$WLUR.
\end{proof}
In  \Cref{eg WLUR quot prox}, we see that, in general, a quotient space of  a $k-$WLUR space need not be $k-$WLUR.  Further, we note that there exists a space $X$ and a subspace $M$ of $X$ such that both $X$ and $X/M$ are $k-$WUR (hence, $k-$WLUR), but $M$ is not proximinal on $X$. To see this, consider any WUR space which is not reflexive (see, \Cref{eg converse}).
\begin{example}\label{eg WLUR quot prox}
	Let $k \in \mathbb{Z}^+$ and $X= (\ell_1, \|\cdot\|_r)$ be the space considered in  \cite[Example 1]{SmTu1990}. That is,  for any $x \in \ell_1$, $\|x\|_r= (\|x\|_1^2+ \|S(x)\|_2^2)^{\frac{1}{2}}$, where $S: \ell_1 \to \ell_2$ defined as $S(\alpha_n)= (\alpha_n2^{\frac{-n}{2}})$ for all $(\alpha_n) \in \ell_1$. By \cite[Theorem 1]{SmTu1990}, $(\ell_1, \|\cdot\|_1) \cong X/M$ for some subspace $M$ of $X$.  Therefore, $X/M$ is not $k-$rotund. However, following \cite[Example 6]{Smit1978a}, it is easy to see that $X$ is LUR (hence, $X$ is $k-$WLUR).
\end{example}
From the \Cref{def k-wur}, it follows that every subspace of a $k-$WUR (respectively, $k-$WLUR) space is $k-$WUR (respectively, $k-$WLUR). Further, in view of \Cref{coro WLUR quot,coro WUR quot}, it is natural to ask whether $k-$WUR and  $k-$WLUR are three space properties \cite[Definition 1.7.8]{Megg1998}  or not. To see this, consider a space $X= M \oplus_1 (\mathbb{R}^k, \|\cdot\|_1)$, where $M$ is a WUR space and $k \in \mathbb{Z^+}.$ Observe that $X$ is not $k-$WLUR. However, $X/M$ and $M$ are $k-$WUR. Thus, $k-$WUR and $k-$WLUR are not three space properties.

	Now we present a result that is closely related to \Cref{coro WLUR quot}, which  also generalizes \cite[Proposition 3.2]{Rmou2017}.
\begin{proposition}
Let $Y$ be a subspace of $X$ such that $Y^\perp \subseteq NA(X),$ where $NA(X)$ is the set of all norm attaining functionals on $X$. If $X$ is  $k-$WLUR, then $X/Y$ is $k-$rotund.
\end{proposition}
\begin{proof} 	Suppose $X/Y$ is not  $k-$rotund. Then there exist $(k+1)$ elements $x_1+Y, x_2+Y, \dots, x_{k+1}+Y$ in $S_{X/Y}$ such that $\|\sum_{i=1}^{k+1}(x_i+Y)\|=k+1,$ but  $D_k[(x_i+Y)_{i=1}^{k+1};(\widetilde{g_j})_{j=1}^k]=\epsilon$ for some $\epsilon >0$ and $\widetilde{g_1}, \widetilde{g_2}, \dots, \widetilde{g_k}  \in S_{(X/Y)^*}.$ By Hahn-Banach Theorem, there exists  $\widetilde{f} \in S_{(X/Y)^*}$ such that $\widetilde{f}(\sum_{i=1}^{k+1}x_i+Y)=k+1.$ Observe that $\widetilde{f}(x_i+Y)=1$ for all $1\leq i\leq k+1.$ Let $T: (X/Y)^* \xrightarrow{} Y^\perp$ be the isometric isomorphism defined by $T(\widetilde{h})=\widetilde{h} \circ q,$ where $q: X \xrightarrow{} X/Y$ is the quotient map. Clearly, $\widetilde{f} \circ q \in Y^\perp \subseteq NA(X)$ and $\|\widetilde{f} \circ q\|=1.$ Thus, there exists $z\in S_X$ such that $(\widetilde{f} \circ q )(z)=\widetilde{f}(z+Y)=1.$
	For every $1\leq i \leq k+1,$ choose a sequence $(y_n^{(i)})$ in $Y$ such that $\|x_i-y_n^{(i)}\| \xrightarrow{} 1.$ Let $\alpha \in \mathcal{S}_{k+1}(k).$ Note that
	\[
	(\widetilde{f} \circ q)\left(\sum_{i=1}^k(x_{\alpha_i}-y_n^{(\alpha_i)})+z\right) = (\widetilde{f} \circ q)\left(\sum_{i=1}^kx_{\alpha_i}+z\right) = \widetilde{f}\left(\left(\sum_{i=1}^kx_{\alpha_i}+z\right)+Y\right)=k+1.
	\]
	Since
	$
	(\widetilde{f} \circ q)\left(\sum_{i=1}^k(x_{\alpha_i}-y_n^{(\alpha_i)})+z \right) \leq \left\Vert \sum_{i=1}^k(x_{\alpha_i}-y_n^{(\alpha_i)})+z \right\Vert \leq \sum_{i=1}^k\|x_{\alpha_i}-y_n^{(\alpha_i)}\|+1,
	$
	we have $\|\sum_{i=1}^k(x_{\alpha_i}-y_n^{(\alpha_i)})+z\|\xrightarrow{} k+1.$ By assumption, we have $D_k[z, (x_{\alpha_i}-y_n^{(\alpha_i)})_{i=1}^k;(f_j)_{j=1}^k] \xrightarrow{} 0$ for all $f_1, f_2, \dots, f_k \in S_{X^*},$ in particular, $D_k[z, (x_{\alpha_i}-y_n^{(\alpha_i)})_{i=1}^k;(\widetilde{g_j} \circ q)_{j=1}^k] \xrightarrow{} 0.$ Thus, it follows that $D_k[z, (x_{\alpha_i})_{i=1}^k;(\widetilde{g_j} \circ q)_{j=1}^k] = 0$ which further implies, $D_k[z+Y, (x_{\alpha_i}+Y)_{i=1}^k;(\widetilde{g_j})_{j=1}^k] =0.$ By \cite[Lemma 2]{Suya2000}, \[ |D_k[(x_{i}+Y)_{i=1}^{k+1};(\widetilde{g_j})_{j=1}^k]| \leq \sum_{\alpha \in \mathcal{S}_{k+1}(k)} |D_k[z+Y,(x_{\alpha_i}+Y)_{i=1}^{k}; (\widetilde{g_j})_{j=1}^k ]| = 0,\]
	which is a contradiction. Hence $X/Y$ is $k-$rotund.
\end{proof}
In the rest of the section, we mainly focus on the finite and infinite $\ell_p-$product of the notions $k-$WUR, $k-$WLUR, $k-$WMLUR.

In \Cref{eg k not k-1 WUR} it is noted that the notions $k-$WUR (respectively, $k-$WLUR, $k-$WMLUR) and $(k+1)-$WUR (respectively, $(k+1)-$WLUR, $(k+1)-$WMLUR)  do not coincide in general. However, we prove that these notions coincide in $\ell_p-$product of a Banach space for $1 < p< \infty$. For this we need the following results.
\begin{theorem}\label{thrm WUR stability}
	Let $1 \leq p \leq \infty,$ $X_i$ be a Banach space for all $ 1 \leq i \leq k$ and $X=(\oplus_pX_i)_{i=1}^{k}.$ Then the following statements hold.
	\begin{enumerate}
		\item If $X$ is $k-$WUR, then $X_i$ is WUR for some $1 \leq i \leq k.$
		\item If $X$ is $k-$WLUR, then $X_i$ is  WLUR for some $1 \leq i \leq k.$
		\item If $X$ is $k-$WMLUR, then $X_i$ is  WMLUR for some $1 \leq i \leq k.$
	\end{enumerate}
\end{theorem}
\begin{proof}
$(1)$: Let $k \geq 2 ,$ $1 \leq p < \infty$ and $X$ be $k-$WUR. Suppose $X_i$ is not WUR for all $1 \leq i \leq k.$ Then for each $1 \leq i \leq k$ there exist $f_i \in S_{X_i^{*}}$ and $(2)-$sequences $(x_n^{(i)}),$ $(y_n^{(i)})$ in $S_{X_i}$ such that $\|x_n^{(i)}+y_n^{(i)}\| \to 2,$ but $|f_i(x_n^{(i)}-y_n^{(i)})|> \epsilon$ for all $n \in \mathbb{N},$ for some  $\epsilon>0.$ For every $n \in \mathbb{N},$ define
	\[
		\begin{aligned}
			z_n^{(1)} &= \frac{1}{k^{\frac{1}{p}}}(
			x_n^{(1)}, x_n^{(2)}, x_n^{(3)}, \dots, x_n^{(k-1)}, x_n^{(k)} ),\\
			z_n^{(2)} &= \frac{1}{k^{\frac{1}{p}}}(
			y_n^{(1)}, x_n^{(2)}, x_n^{(3)}, \dots, x_n^{(k-1)}, x_n^{(k)}),\\
			z_n^{(3)} &= \frac{1}{k^{\frac{1}{p}}}(
			y_n^{(1)}, y_n^{(2)}, x_n^{(3)}, \dots, x_n^{(k-1)}, x_n^{(k)}), \\
			\vdots & \\
			z_n^{(k+1)} &= \frac{1}{k^{\frac{1}{p}}}(
			y_n^{(1)}, y_n^{(2)}, y_n^{(3)}, \dots, y_n^{(k-1)}, y_n^{(k)}).
			\end{aligned}
	\]
	Clearly,  $z_n^{(t)} \in S_X$  for all $1 \leq t\leq k+1$ and $n \in \mathbb{N}.$ Now for every $1 \leq j \leq k,$ let $g_j= (0,\dots,0, f_j,0, \dots 0) \in S_{X^*},$ here $f_j$ is in the $j^{th}$ coordinate. Since
	\[
	\begin{aligned}
	D_k[(z_n^{(t)})_{t=1}^{k+1}; (g_j)_{j=1}^{k}] &= \frac{1}{k^{\frac{k}{p}}}
		\begin{vmatrix}
			1 & 1 & 1& \dots & 1 \\
			f_1(x_n^{(1)}) & f_1(y_n^{(1)}) & f_1(y_n^{(1)}) & \dots & f_1(y_n^{(1)})\\
				f_2(x_n^{(2)}) & f_2(x_n^{(2)}) &  f_2(y_n^{(2)}) & \dots & f_2(y_n^{(2)})\\
			\vdots & \vdots & \vdots & \ddots & \vdots \\
			f_k(x_n^{(k)}) & f_k(x_n^{(k)}) &  f_k(x_n^{(k)}) &  \dots & f_k(y_n^{(k)})
		\end{vmatrix}\\
	&= \frac{1}{k^{\frac{k}{p}}}
	\begin{vmatrix}
	0 & 0 & 0& \dots & 1 \\
	f_1(x_n^{(1)}-y_n^{(1)}) & 0 & 0 & \dots & f_1(y_n^{(1)})\\
	f_2(x_n^{(2)}-y_n^{(2)}) & f_2(x_n^{(2)}-y_n^{(2)}) &  0 & \dots & f_2(y_n^{(2)})\\
	\vdots & \vdots & \vdots & \ddots & \vdots \\
	f_k(x_n^{(k)}-y_n^{(k)}) & f_k(x_n^{(k)}-y_n^{(k)}) &  f_k(x_n^{(k)}-y_n^{(k)}) &  \dots & f_k(y_n^{(k)})
	\end{vmatrix},
	\end{aligned}
\]
it follows that
\[
|D_k[(z_n^{(t)})_{t=1}^{k+1}; (g_j)_{j=1}^{k}]| = \frac{1}{k^{\frac{k}{p}}}\left\vert \prod\limits_{i=1}^{k}f_i(x_n^{(i)}-y_n^{(i)})\right\vert > \frac{1}{k^{\frac{k}{p}}} \epsilon^{k}.
\]	
 Since
\[
\begin{aligned}
	\frac{1}{k+1}\left\Vert\sum\limits_{t=1}^{k+1}z_n^{(t)}\right\Vert&= \frac{1}{(k+1)k^{\frac{1}{p}}} \left(\sum\limits_{i=1}^{k}\left\Vert ix_n^{(i)}+(k+1-i)y_n^{(i)}\right \Vert ^{p}\right)^{\frac{1}{p}}	\\
	&= \frac{1}{k^{\frac{1}{p}}} \left(\sum\limits_{i=1}^{k}\left\Vert\frac{i}{k+1}x_n^{(i)}+\left( 1- \frac{i}{k+1} \right) y_n^{(i)}\right\Vert^{p}\right)^{\frac{1}{p}},
\end{aligned}
\]
we have $\|\sum_{t=1}^{k+1}z_n^{(t)}\| \to k+1.$
This is a contradiction. For the case $p= \infty,$ a similar proof holds.\\
$(2)$: Let $k \geq 2 ,$ $1 \leq p < \infty$ and $X$ be $k-$WLUR. Suppose $X_i$ is not WLUR for all $1 \leq i \leq k.$ Then for each $1 \leq i \leq k$ there exist $f_i \in S_{X_i^{*}}$, $y_i\ \in S_{X_i}$ and  a sequence $(x_n^{(i)})$ in $S_{X_i}$ such that $\|y_i+x_n^{(i)}\| \to 2,$ but $|f_i(y_i-x_n^{(i)})|> \epsilon$ for all $n \in \mathbb{N},$ for some $\epsilon>0.$  Now, using the preceding functionals, sequences and by assuming $y_n^{(i)}=y_i$ for all $n \in \mathbb{N}$ and $1 \leq i \leq k$,  construct $k-$functionals $g_1, g_2, \dots, g_k$ in $S_{X^*}$ and  $k-$sequences $(z_n^{(1)}), (z_n^{(2)}), \dots, (z_n^{(k)})$ in $S_{X}$ as in the proof of $(1).$  Let $z= \frac{1}{k^{\frac{1}{p}}}(y_1, y_2, \dots, y_k) \in S_{X}.$ By following the similar technique as in the proof of $(1),$ we have $\|z+ \sum_{i=1}^{k}z_n^{(i)}\| \to k+1$ and $|D_k[z, (z_n^{(i)})_{i=1}^{k}; (g_j)_{j=1}^{k}]| > \frac{1}{k^{\frac{k}{p}}} \epsilon^{k}$ for all $n \in \mathbb{N}$. This is a contradiction. For the case $p= \infty,$ a similar proof holds.\\
$(3)$: Let $k \geq 2 ,$ $1 \leq p < \infty$ and $X$ be $k-$WMLUR. Suppose $X_i$ is not WMLUR for all $1 \leq i \leq k.$ Therefore, by \Cref{thrm MLUR B_X}, $B_{X_i}$ is not $w$SCh on $2S_{X_i}$  for all $1 \leq i \leq k.$ Then for each $1 \leq i \leq k,$ there exist $f_i \in S_{X_i^{*}},$ $w_i \in 2S_{X_i}$ and $(2)-$sequences $(x_n^{(i)}),$ $(y_n^{(i)})$ in $B_{X_i}$ such that $\|x_n^{(i)}-w_i\| \to 1$ and $\|y_n^{(i)}-w_i\| \to 1,$ but $|f_i(x_n^{(i)}-y_n^{(i)})|> \epsilon$ for all $n \in \mathbb{N},$ for some  $\epsilon>0.$ Now, using the  preceding sequences and functionals, we define $(k+1)-$sequences $(z_n^{(1)}), (z_n^{(2)}), \dots, (z_n^{(k+1)})$ in $B_X$ and $(k)-$functionals $g_1, g_2, \dots, g_k$ in $S_{X^*}$ as in the proof of $(1).$ Let $w= \frac{1}{k^{\frac{1}{p}}}(w_1, w_2, \dots, w_k) \in 2S_X.$ Since for any $1 \leq t \leq k+1,$ we have
\[
1= d(w, B_X) \leq \|z_n^{(t)}-w\|= \frac{1}{k^{\frac{1}{p}}} \left( \sum\limits_{j=1}^{t-1} \|y_n^{(j)}-w_j\|^p+ \sum_{j=t}^{k}\|x_n^{(j)}-w_j\|^{p}\right)^{\frac{1}{p}}
\]
and hence $\|z_n^{(t)}-w\| \to d(w, B_X).$ Using the similar argument involved  in  the proof of $(1)$, we have
\[
|D_k[(z_n^{(t)})_{t=1}^{k+1}; (g_j)_{j=1}^{k}]| = \frac{1}{k^{\frac{k}{p}}}\left\vert \prod\limits_{i=1}^{k}f_i(x_n^{(i)}-y_n^{(i)})\right\vert > \frac{1}{k^{\frac{k}{p}}} \epsilon^{k}.
\]	
Thus, $B_X$ is not $k-$$w$SCh at $w.$ Therefore, by \Cref{thrm MLUR B_X}, $X$ is not $k-$WMLUR, which is a contradiction. For the case $p= \infty,$ a similar proof holds. Hence the proof.
\end{proof}
We notice that \Cref{thrm WUR stability} can be extended to infinite $\ell_p-$product  which is presented in the following  corollary.
\begin{corollary}\label{coro WUR stability}
	Let $1 \leq p \leq  \infty,$ $X_i$ be a Banach space for all $i \in \mathbb{N}$ and $X= (\oplus_pX_i)_{i \in \mathbb{N}}.$ If $X$ is $k-$WUR (respectively, $k-$WLUR, $k-$WMLUR), then all but except $(k-1)-$spaces of the collection $\{X_i\}_{i \in \mathbb{N}}$ are WUR (respectively, WLUR, WMLUR).
\end{corollary}

The following corollary is an immediate consequence of \Cref{prop k implies k+1}, \Cref{coro WUR stability} and  \cite[A.2, A.3, A.4]{Smit1986}.
\begin{corollary}\label{coro lp WUR}
	Let $1 < p < \infty.$ Then the following statements are equivalent.
	\begin{enumerate}
		\item $X$ is WUR (respectively, WLUR, WMLUR).
		\item $\ell_p(X)$ is WUR (respectively, WLUR, WMLUR).
		\item $\ell_p(X)$ is $k-$WUR (respectively, $k-$WLUR, $k-$WMLUR).
	\end{enumerate}
\end{corollary}
From the preceding result, we conclude that unlike the notion WUR (respectively, WLUR, WMLUR), the notion $k-$WUR (respectively, $k-$WLUR, $k-$WMLUR) for $k>1$, need not be lifted to $\ell_p-$ product space. To see this consider a space $X$ which is $k-$WUR but not rotund (see, \Cref{eg k not k-1 WUR}).

Now, we present a necessary condition for a finite $\ell_p-$product space to be $k-$WUR (respectively, $k-$WLUR, $k-$WMLUR).
\begin{theorem}\label{thrm WUR l_p finite for}
	Let $X,$ $Y$ be  Banach spaces and $1 \leq p \leq \infty.$ For any $k \in \mathbb{Z}^{+},$ there exist $k_1, k_2 \in \mathbb{Z}^{+}$ with $k=k_1+k_2-1$ such that the following statements hold.
	\begin{enumerate}
		\item If $X \oplus_p Y$ is $k-$WUR, then $X$ is $k_1-$WUR  and $Y$ is $k_2-$WUR.
		\item If $X \oplus_p Y$ is $k-$WLUR, then $X$ is $k_1-$WLUR  and $Y$ is $k_2-$WLUR.
		\item If $X \oplus_p Y$ is $k-$WMLUR, then $X$ is $k_1-$WMLUR  and $Y$ is $k_2-$WMLUR.
	\end{enumerate}
\end{theorem}
\begin{proof}
	$(1)$: Let  $X \oplus_{p} Y$ be $k-$WUR space. Suppose $Y$ is WUR, then there is nothing to prove. Assume $Y$ is not WUR. Then there exists $k_2 \in \mathbb{Z}^{+}$ such that  $2 \leq  k_2 \leq k$ and  $Y$ is $k_2-$WUR, but   not $(k_2-1)-$WUR. Now, it is enough to show that $X$ is $k_1-$WUR, where $k_1= k-k_2+1.$ Suppose $X$ is not $k_1-$WUR. Then there exist $\epsilon >0,$ $(k_1+1)-$sequences $(x_n^{(1)}), (x_n^{(2)}),$ $\dots,$$(x_n^{(k_1+1)})$ in $S_X$ with $\|\sum_{i=1}^{k_1+1}x_n^{(i)}\| \to k_1+1,$ but $|D_{k_1}[(x_n^{(i)})_{i=1}^{k_1+1}; (f_j)_{j=1}^{k_1}]| \geq \epsilon$ for all $n \in \mathbb{N}$ and for some $f_1, f_2, \dots, f_{k_1} \in S_{X^*}.$ Since $Y$ is not $(k_2-1)-$WUR, there exist $(k_2)-$sequences $(y_n^{(1)}),$ $(y_n^{(2)}), \dots,$$(y_n^{(k_2)})$ in $S_Y$ with $\|\sum_{i=1}^{k_2}y_n^{(i)}\| \to k_2,$ but $|D_{k_2-1}[(y_n^{(i)})_{i=1}^{k_2}; (g_j)_{j=1}^{k_2-1}] |\geq \epsilon$ for all $n \in \mathbb{N}$ and for some $g_1, g_2, \dots, g_{k_2-1} \in S_{Y^*}.$ Choose $r>0$ with $\|(r,r)\|_{p}=1.$  For every $n \in \mathbb{N}$ and $1 \leq i \leq k+1,$ define
	\[
	z_n^{(i)}=  \begin{cases}
		r(x_n^{(i)}, y_n^{(k_2)}), & \ \textnormal{if} \ 1 \leq i \leq k_1+1;\\
		r(x_n^{(k_1+1)}, y_n^{(i-k_1-1)}), & \ \textnormal{if} \ k_1+2 \leq i \leq k_1+k_2.	
	\end{cases}
	\]
Clearly, $z_n^{(i)} \in X\oplus_p Y$ and $\|z_n^{(i)}\|=1$ for all $1 \leq i \leq  k+1,$ $n \in \mathbb{N}.$ Note that
\begin{small}
\[
	\begin{aligned}
		1- \frac{1}{k+1}\left\Vert \sum\limits_{i=1}^{k+1}z_n^{(i)}\right\Vert &= \|(r,r)\|_p- \frac{r}{k+1} \left\Vert \left( \sum\limits_{i=1}^{k_1+1}x_n^{(i)}+(k_2-1)x_n^{(k_1+1)}, \sum\limits_{i=1}^{k_2-1}y_n^{(i)}+ (k_1+1)y_n^{(k_2)}
		\right)\right\Vert\\
		 &= \|(r,r)\|_p- \frac{r}{k+1} \left\Vert \left(\left\Vert \sum\limits_{i=1}^{k_1+1}x_n^{(i)}+(k_2-1)x_n^{(k_1+1)}\right\Vert, \left\Vert \sum\limits_{i=1}^{k_2-1}y_n^{(i)}+ (k_1+1)y_n^{(k_2)} \right\Vert
		\right)\right\Vert\\
		& \leq  r \left\Vert\left(1- \frac{1}{k+1}\left\Vert\sum\limits_{i=1}^{k_1}x_n^{(i)}+ (k_2)x_n^{(k_1+1)}\right\Vert, 1- \frac{1}{k+1}\left\Vert\sum\limits_{i=1}^{k_2-1}y_n^{(i)}+ (k_1+1)y_n^{(k_2)}\right\Vert \right)\right \Vert\\
		&\leq  r \left\vert 1- \frac{1}{k+1} \left\Vert \sum\limits_{i=1}^{k_1}x_n^{(i)}+  (k_2)x_n^{(k_1+1)}\right\Vert\right \vert+
			r \left\vert 1- \frac{1}{k+1} \left\Vert \sum\limits_{i=1}^{k_2-1}y_n^{(i)}+ (k_1+1)y_n^{(k_2)}\right\Vert\right \vert.
		\end{aligned}
\]	
\end{small}	
Thus, by \cite[Lemma 3.8]{GaTh2023}, it follows that  $\frac{1}{k+1}\| \sum_{i=1}^{k+1}z_n^{(i)}\| \to 1.$ For every $1 \leq j \leq k,$ define
\[
h_j= \begin{cases}
	(f_j,0), & \ \textnormal{if} \ 1 \leq j \leq k_1 ;\\
	(0, g_{j-k_1}), & \ \textnormal{if} \ k_1+1 \leq j \leq k_1+k_2-1.
\end{cases}
\]
Clearly, $h_j \in (X \oplus_p Y)^{*}$ and $\|h_j\|=1$ for all $1 \leq j \leq  k.$  Now, consider
\[
	\begin{aligned}
		D_k[(z_n^{(i)})_{i=1}^{k+1}; (h_j)_{j=1}^{k}]&=  D_{k}[((rx_n^{(i)},0))_{i=1}^{k_1+1}, (r(x_n^{(k_1+1)}, y_n^{(l)}-y_n^{(k_2)}))_{l=1}^{k_2-1}; (h_j)_{j=1}^{k}] \\
		&=  r^{k} det \left(\begin{bmatrix}
			A_n & B_n\\
			0 & C_n
		\end{bmatrix}\right)\hspace{-0.1cm},
	\end{aligned}
\]
where $A_n= [a_{i,j}^{(n)}]$, here  $a_{1,j}^{(n)}=1,$ $a_{i+1,j}^{(n)}=f_i(x_n^{(j)})$  for all $1 \leq i \leq k_1,$ $1 \leq j \leq k_1+1$; \\
$ B_n=[b_{l,m}^{(n)}]$, here $ b_{1,m}^{(n)}=1, b_{l+1,m}^{(n)}= f_l(x_n^{k_1+1})$ for all $1 \leq m \leq k_2-1$, $1 \leq l \leq k_1$;\\ $C_n=[c_{s,t}^{(n)}],$ here $c_{s,t}^{(n)}= g_s(y_n^{t}-y_n^{k_2})$ for all $1 \leq s,t \leq k_2-1.$\\ Therefore, by assumption, for all $n \in \mathbb{N}$ we have
\[
		\vert D_k[(z_n^{(i)})_{i=1}^{k+1}; (h_j)_{j=1}^{k}] \vert =  r^{k} |D_{k_1}[(x_n^{(i)})_{i=1}^{k_1+1}; (f_j)_{j=1}^{k_1}] D_{k_2-1}[(y_n^{(i)})_{i=1}^{k_2}; (g_j)_{j=1}^{k_2-1}]|
		 \geq  r^{k} \epsilon^{2},
\]
which is a contradiction to $X\oplus_p Y$ is $k-$WUR. Thus, $X$ is $k_1-$WUR. \\
$(2)$: 	Let  $X \oplus_{p} Y$ be  $k-$WLUR space. Suppose $Y$ is WLUR, then there is nothing to prove. Assume $Y$ is not WLUR. Then there exists $k_2 \in \mathbb{Z}^{+}$ such that  $2 \leq  k_2 \leq k$ and  $Y$ is $k_2-$WLUR, but   not $(k_2-1)-$WLUR. Now, it is enough to show that $X$ is $k_1-$WLUR, where $k_1= k-k_2+1.$ Suppose $X$ is not $k_1-$WLUR. Then there exist $x \in S_X$, $\epsilon >0$ and $(k_1)-$sequences $(x_n^{(1)}), (x_n^{(2)}),\dots,(x_n^{(k_1)})$ in $S_X$ with $\|x+\sum_{i=1}^{k_1}x_n^{(i)}\| \to k_1+1,$ but $|D_{k_1}[x,(x_n^{(i)})_{i=1}^{k_1}; (f_j)_{j=1}^{k_1}]| \geq \epsilon$ for all $n \in \mathbb{N}$ and for some $f_1, f_2, \dots, f_{k_1} \in S_{X^*}$. Since $Y$ is not $(k_2-1)-$WLUR, there exist $y \in S_X$ and $(k_2-1)-$sequences $(y_n^{(1)}),$ $(y_n^{(2)}), \dots,$$(y_n^{(k_2-1)})$ in $S_Y$ with $\|y+\sum_{i=1}^{k_2-1}y_n^{(i)}\| \to k_2$, but $|D_{k_2-1}[y,(y_n^{(i)})_{i=1}^{k_2-1}; (g_j)_{j=1}^{k_2-1}] |\geq \epsilon$ for all $n \in \mathbb{N}$ and for some $g_1, g_2, \dots, g_{k_2-1} \in S_{Y^*}$. Using preceding functionals, consider $k-$functionals $h_1, h_2, \dots, h_{k}$ in $S_{(X \oplus_p Y)^*}$ as in the proof of $(1)$. Choose $r>0$ with $\|(r,r)\|_{p}=1.$ Let $z= r(x, y)$.  For every $n \in \mathbb{N}$ and $1 \leq i \leq k$, define
	\[
	z_n^{(i)}=  \begin{cases}
		r(x_n^{(i)}, y), & \ \textnormal{if} \ 1 \leq i \leq k_1;\\
		r(x, y_n^{(i-k_1)}), & \ \textnormal{if} \ k_1+1 \leq i \leq k_1+k_2-1.			
	\end{cases}
	\]
	Clearly, $z, z_n^{(i)} \in S_{(X\oplus_p Y)}$ for all $1 \leq i \leq  k$ and $n \in \mathbb{N}$. Now using similar argument as in the proof of $(1)$, we have  $\frac{1}{k+1}\|z+ \sum_{i=1}^{k}z_n^{(i)}\| \to 1$ and
		$	\vert D_k[z,(z_n^{(i)})_{i=1}^{k}; (h_j)_{j=1}^{k}] \vert  \geq  r^{k} \epsilon^{2}$ for all $n \in \mathbb{N}$. This contradicts the assumption $X\oplus_p Y$ is $k-$WLUR. Thus, $X$ is $k_1-$WLUR. \\
$(3)$: Let $X \oplus_{p} Y$ be  $k-$WMLUR space. Suppose $Y$ is WMLUR, then there is nothing to prove. Assume $Y$ is not WMLUR. Then there exists $k_2 \in \mathbb{Z}^{+}$ such that  $2 \leq  k_2 \leq k$ and  $Y$ is $k_2-$WMLUR, but   not $(k_2-1)-$WMLUR. Now, it is enough to show that $X$ is $k_1-$WMLUR, where $k_1= k-k_2+1.$ Suppose $X$ is not $k_1-$WMLUR. Then, by \Cref{thrm MLUR B_X}, $B_X$ is not $k_1-$$w$SCh on $2S_X.$ Therefore there exist $x \in S_X,$ $\epsilon >0$ and $(k_1+1)-$sequences $(x_n^{(1)}), (x_n^{(2)}),\dots,(x_n^{(k_1+1)})$ in $B_X$ with $\|x_n^{(i)}-2x\| \to 1,$ but $|D_{k_1}[(x_n^{(i)})_{i=1}^{k_1+1}; (f_j)_{j=1}^{k_1}]| \geq \epsilon$ for all $n \in \mathbb{N}$ and  for some $f_1, f_2, \dots, f_{k_1} \in S_{X^*}.$ Since, by \Cref{thrm MLUR B_X}, $B_Y$ is not $(k_2-1)-$$w$SCh on $2S_Y,$ there exist $y \in S_Y$ and $(k_2)-$sequences $(y_n^{(1)}),(y_n^{(2)}), \dots,(y_n^{(k_2)})$ in $B_Y$ with $\|y_n^{(i)}-2y\| \to 1,$ but $|D_{k_2-1}[(y_n^{(i)})_{i=1}^{k_2}; (g_j)_{j=1}^{k_2-1}] |\geq \epsilon$ for all $n \in \mathbb{N}$ and for some $g_1, g_2, \dots, g_{k_2-1} \in S_{Y^*}.$ Choose $r>0$ with $\|(r,r)\|_{p}=1.$ Using preceding sequences and functionals, consider $(k+1)-$sequences $(z_n^{(1)}), (z_n^{(2)}), \dots, (z_n^{(k+1)})$ in $B_{X \oplus_p Y}$ and $k-$functionals $h_1, h_2, \dots, h_k$ in $S_{(X \oplus_p Y)^*}$  as in the proof of $(1)$. Let $z=r(x,y) \in S_{X \oplus_p Y}$. It is easy to verify that, $\|z_n^{(i)}-2z\| \to 1$ for all $1 \leq i \leq k+1$. Now, following the similar technique, as in the proof of $(1)$, we get
$\vert D_k[(z_n^{(i)})_{i=1}^{k+1}; (h_j)_{j=1}^{k}] \vert \geq  r^{k} \epsilon^{2}$ for all $n \in \mathbb{N}$, which implies $B_{X\oplus_p Y}$ is not $k-$$w$SCh on $X\oplus_p Y.$ Therefore, by \Cref{thrm MLUR B_X}, $X$ is not $k-$WMLUR, which is a contradiction. Thus, $X$ is $k_1-$WMLUR. Hence the proof.
\end{proof}
The next result is an immediate consequence of  \Cref{thrm WUR l_p finite for} and the fact  $(\oplus_pX_i)_{i=1}^{d} \cong (\oplus_pX_i)_{i=1}^{d-1} \oplus_p X_{d}.$
\begin{corollary}\label{coro WUR l_p finite for}
	Let $d \in \mathbb{Z}^{+}$,  $d>1$ and $1 \leq p \leq \infty.$ Let $X_i$ be a Banach space for all $1 \leq i \leq d$ and $X= (\oplus_pX_i)_{i=1}^{d}.$ If $X$ is $k-$WUR (respectively, $k-$WLUR, $k-$WMLUR), then there exist $k_1, k_2, \dots, k_d \in \mathbb{Z}^{+}$ such that $k= \sum_{i=1}^{d}k_i-d+1$ and $X_i$ is $k_i-$WUR (respectively, $k_i-$WLUR, $k_i-$WMLUR).
\end{corollary}
In the following result, we provide a sufficient condition for a finite $\ell_p-$product space  to be $k-$WUR (respectively, $k-$WLUR, $k-$WMLUR).
\begin{theorem}\label{thrm WUR l_p finite rev}
	Let $X, Y$ be Banach spaces and $1<p < \infty.$ For any $k_1, k_2, k \in \mathbb{Z}^{+}$ satisfying $k=k_1+k_2-1$, the following statements hold.
	\begin{enumerate}
		\item If $X$ is $k_1-$WUR  and $Y$ is $k_2-$WUR, then $X \oplus_p Y$ is $k-$WUR.
		\item If $X$ is $k_1-$WLUR  and $Y$ is $k_2-$WLUR, then $X \oplus_p Y$ is $k-$WLUR.
		\item If $X$ is $k_1-$WMLUR  and $Y$ is $k_2-$WMLUR, then $X \oplus_p Y$ is $k-$WMLUR.
	\end{enumerate}
\end{theorem}

\begin{proof}
$(1)$: Let $X$ be $k_1-$WUR, $Y$ be $k_2-$WUR and $k= k_1+k_2-1.$ Let $(z_n^{(1)}), (z_n^{(2)}), \dots, (z_n^{(k+1)})$ be $(k+1)-$sequences in $S_{(X\oplus_p Y)}$ with $\|\sum_{i=1}^{k+1}z_n^{(i)}\| \to k+1$ and $h_1, h_2, \dots, h_{k} \in S_{(X\oplus_pY)^{*}}.$ Clearly $h_j= (f_j, g_j)$  for some $f_j \in B_{X^*},$ $g_j \in B_{Y^*},$ for all $1 \leq j \leq k$ and  $z_n^{(i)}= (x_n^{(i)}, y_n^{(i)})$ for some $x_n^{(i)} \in B_X,$ $y_n^{(i)} \in B_Y$ for all $n \in \mathbb{N},$ $1 \leq i \leq k+1.$ Let $1 \leq i < j \leq k+1.$ Since
\[
\frac{1}{k+1}\left\Vert\sum\limits_{t=1}^{k+1}z_n^{(t)} \right\Vert  \leq \frac{1}{k+1} \left(\left \Vert z_n^{(i)} +z_n^{(j)}\right\Vert + k-1\right) \leq 1,
\]
it follows that $\|z_n^{(i)}+z_n^{(j)}\| \to 2.$ Note that
\[
	\begin{aligned}
	\Vert z_n^{(i)}+ z_n^{(j)}\Vert &= \Vert (\Vert x_n^{(i)}+x_n^{(j)}\Vert, \Vert y_n^{(i)}+y_n^{(j)}\Vert) \Vert\\
	& \leq  \Vert (\Vert x_n^{(i)} \Vert + \Vert x_n^{(j)}\Vert, \Vert y_n^{(i)} \Vert +\Vert y_n^{(j)}\Vert) \Vert\\
	&=  \Vert (\Vert x_n^{(i)} \Vert , \Vert y_n^{(i)} \Vert )+ ( \Vert x_n^{(j)}\Vert,  \Vert y_n^{(j)}\Vert) \Vert\\
	&\leq 2.
	\end{aligned}
\]
Therefore,  $\Vert (\Vert x_n^{(i)} \Vert , \Vert y_n^{(i)} \Vert )+ ( \Vert x_n^{(j)}\Vert,  \Vert y_n^{(j)}\Vert) \Vert \to 2.$ Since $(\mathbb{R}^{2}, \|\cdot\|_p)$ is uniformly rotund, it follows that
$\Vert( \Vert x_n^{(i)}\Vert- \Vert x_n^{(j)} \Vert, \Vert y_n^{(i)} \Vert-\Vert y_n^{(j)} \Vert ) \Vert \to 0,$ which implies  $\Vert x_n^{(i)}\Vert- \Vert x_n^{(j)} \Vert \to 0$ and
$\Vert y_n^{(i)} \Vert-\Vert y_n^{(j)} \Vert \to 0.$\\
Case$-(i)$: Assume that the sequence $(\|x_n^{(1)}\|)$ converges. Therefore, for every $1 \leq i \leq k+1$ we have $\|x_n^{(i)}\| \to a_1$ for some $a_1 \in [0,1]$, which further implies  $\|y_n^{(i)}\| \to a_2$, where $a_2= (1-a_1^{p})^{\frac{1}{p}} .$ Let $\alpha \in \mathcal{S}_{k+1}(k_1+1).$
Note that for any subsequence $(n_m)$ of $(n)$,   we have   $\| \sum_{i=1}^{k_1+1} z_{n_m}^{(\alpha_i)} \| \to k_1+1$ and
\[
	\begin{aligned}
		\limsup\limits_{n \to \infty} \left \Vert \sum\limits_{i=1}^{k_1+1} z_{n_m}^{(\alpha_i)} \right\Vert &= 	\limsup\limits_{n \to \infty}  \left \Vert \left(\left\Vert\sum\limits_{i=1}^{k_1+1} x_{n_m}^{(\alpha_i)} \right\Vert, \left\Vert \sum\limits_{i=1}^{k_1+1} y_{n_m}^{(\alpha_i)} \right\Vert \right)\right\Vert\\
		& \leq \left \Vert \left( \limsup\limits_{n \to \infty} \left\Vert\sum\limits_{i=1}^{k_1+1} x_{n_m}^{(\alpha_i)} \right\Vert,  \limsup\limits_{n \to \infty} \left\Vert \sum\limits_{i=1}^{k_1+1} y_{n_m}^{(\alpha_i)} \right\Vert \right)\right\Vert\\
		& \leq \left \Vert \left( \limsup\limits_{n \to \infty} \sum\limits_{i=1}^{k_1+1} \left\Vert x_{n_m}^{(\alpha_i)} \right\Vert,  \limsup\limits_{n \to \infty}  \sum\limits_{i=1}^{k_1+1}  \left\Vert y_{n_m}^{(\alpha_i)} \right\Vert \right)\right\Vert\\
		&= \|((k_1+1)a_1, (k_1+1)a_2)\|\\
		&= k_1+1.
	\end{aligned}
\]
Thus $\limsup\limits_{n \to \infty} \|\sum_{i=1}^{k_1+1} x_{n_m}^{(\alpha_i)} \| = (k_1+1)a_1,$ which further implies $\|\sum_{i=1}^{k_1+1} x_n^{(\alpha_i)} \|  \to (k_1+1)a_1.$  If $a_1=0,$ by \Cref{lem samir}, we have $D_{k_1}[(x_n^{(\alpha_i)})_{i=1}^{k_1+1}; (f_{\lambda_j})_{j=1}^{k_1}]  \to 0$ for all  $\lambda \in \mathcal{S}_{k}(k_1).$ Suppose $a_1 \neq 0.$  Since $X$ is $k_1-$WUR, we have  $D_{k_1}[(x_n^{(\alpha_i)})_{i=1}^{k_1+1}; (f_{\lambda_j})_{j=1}^{k_1}]  \to 0$ for all  $\lambda \in \mathcal{S}_{k}(k_1).$ Similarly, for every $\beta \in \mathcal{S}_{k+1}(k_2+1)$ we have $\|\sum_{j=1}^{k_2+1} y_n^{(\beta_j)} \| \to (k_2+1)a_2$ and $D_{k_2}[(y_n^{(\beta_i)})_{i=1}^{k_2+1}; (g_{\mu_j})_{j=1}^{k_2}]  \to 0$ for all  $\mu \in \mathcal{S}_{k}(k_2).$
Consider,
\[
\begin{aligned}
	D_k[(z_n^{(i)})_{i=1}^{k+1}; (h_j)_{j=1}^{k}] &= \begin{vmatrix}
		1 & 1& \dots &1\\
		f_1(x_n^{(1)})+g_1(y_n^{(1)}) & f_1(x_n^{(2)})+g_1(y_n^{(2)}) & \dots & f_1(x_n^{(k+1)})+g_1(y_n^{(k+1)})\\
		\vdots & \vdots & \ddots & \vdots\\
		f_k(x_n^{(1)})+g_k(y_n^{(1)}) & f_k(x_n^{(2)})+g_k(y_n^{(2)}) & \dots & f_k(x_n^{(k+1)})+g_k(y_n^{(k+1)})\\
	\end{vmatrix}.
	\end{aligned}
\]
Since the  determinant is multilinear, we can write the preceding determinant as the sum of $2^{k}$ determinants each of order $(k+1).$ Then, by rearranging the rows, we can rewrite all $2^{k}$ determinants such that
\[
|	D_k[(z_n^{(i)})_{i=1}^{k+1}; (h_j)_{j=1}^{k}]| \leq \sum_{t=1}^{2^k}|a_n^{(t)}|,
\]
where
  \[
  	\begin{aligned}
  		a_n^{(t)} &= \begin{vmatrix}
  			1 & 1& \dots &1\\
  			f_{\alpha_1}(x_n^{(1)})& f_{\alpha_1}(x_n^{(2)})& \dots & f_{\alpha_1}(x_n^{(k+1)})\\
  			\vdots & \vdots & \ddots & \vdots\\
  			f_{\alpha_{r_t}}(x_n^{(1)}) & f_{\alpha_{r_t}}(x_n^{(2)}) & \dots & f_{\alpha_{r_t}}(x_n^{(k+1)})\\
  			g_{\beta_1}(y_n^{(1)})& g_{\beta_1}(y_n^{(2)}) & \dots & g_{\beta_1}(y_n^{(k+1)})\\
  				\vdots & \vdots & \ddots & \vdots\\
  			g_{\beta_{s_t}}(y_n^{(1)})& g_{\beta_{s_t}}(y_n^{(2)}) & \dots & g_{\beta_{s_t}}(y_n^{(k+1)})
  		\end{vmatrix}
  	\end{aligned}
  \]
for some $\alpha \in \mathcal{S}_k(r_t),$ $\beta \in \mathcal{S}_k(s_t)$ and  $0 \leq r_t, s_t \leq k$ with $r_t+s_t=k$ for all $1 \leq t \leq 2^{k}.$ Observe that in each determinant $a_n^{(t)},$ either $r_t \geq k_1$ or $s_t \geq k_2.$ Consider the determinant $a_n^{(t_0)},$ for some $1 \leq t_0 \leq 2^{k}.$\\
subcase$-(a)$: Suppose $r_{t_0}\geq k_1.$ Then evaluate the determinant $a_n^{(t_0)}$ using the Laplace expansion of the determinant \cite{HoJo2013} (by fixing the first $(k_1+1)-$rows). Since each entry of the determinant $a_n^{(t_0)}$ is bounded by 1 and $D_{k_1}[(x_n^{(\alpha_i)})_{i=1}^{k_1+1}; (f_{\lambda_j})_{j=1}^{k_1}]  \to 0$ for all $\alpha \in \mathcal{S}_{k+1}(k_1+1),$ $\lambda \in \mathcal{S}_{k}(k_1),$ it follows that $|a_n^{(t_0)}| \to 0.$\\
subcase$-(b)$: Suppose $s_{t_0}\geq k_2.$ Then evaluate the determinant $a_n^{(t_0)}$ using the Laplace expansion of the determinant \cite{HoJo2013} (by fixing the rows $R_1, R_{r_{t_0}+1}, R_{r_{t_0}+2}, \dots, R_{r_{t_0}+k_2}).$ Since each entry of the determinant $a_n^{(t_0)}$ is bounded by 1 and  $D_{k_2}[(y_n^{(\beta_i)})_{i=1}^{k_2+1}; (g_{\mu_j})_{j=1}^{k_2}]  \to 0$ for all $\beta \in \mathcal{S}_{k+1}(k_2+1),$ $\mu \in \mathcal{S}_{k}(k_2),$ it follows that $|a_ n^{(t_0)}| \to 0.$\\
Therefore, $|a_n^{(t)}| \to 0$ for all $1 \leq t \leq 2^{k}.$ Thus, $|D_k[(z_n^{(i)})_{i=1}^{k+1}; (h_j)_{j=1}^{k}]| \to 0.$\\
Case$-(ii)$: Assume that the sequence $(\|x_n^{(1)}\|)$ does not converge. We need to  show that $|D_k[(z_n^{(i)})_{i=1}^{k+1}; (h_j)_{j=1}^{k}]| \to 0.$ Suppose $|D_k[(z_n^{(i)})_{i=1}^{k+1}; (h_j)_{j=1}^{k}]|$ does not converge to $0.$ Then there exist a subsequence $(n_m)$ of $(n)$ and $\epsilon>0$ such that $|D_k[(z_{n_m}^{(i)})_{i=1}^{k+1}; (h_j)_{j=1}^{k}]| \geq \epsilon$ for all $m \in \mathbb{N}.$ Since the sequence $(\|x_{n_m}^{(1)}\|)$ is bounded, there exists a subsequence $(\|x_{m_s}^{(1)}\|)$ of $(\|x_{n_m}^{(1)}\|)$ such that $\|x_{m_s}^{(1)}\| \to b_1$ for some $b_1 \in [0,1].$ Now, by Case$-(i),$ we have $|D_k[(z_{m_s}^{(i)})_{i=1}^{k+1}; (h_j)_{j=1}^{k}]| \to 0$ as $s \to \infty,$ which is a contradiction. Thus,  $|D_k[(z_n^{(i)})_{i=1}^{k+1}; (h_j)_{j=1}^{k}]| \to 0.$ \\
$(2)$: 	Let $X$ be $k_1-$WLUR, $Y$ be $k_2-$WLUR and $k= k_1+k_2-1.$ Let $z \in S_{(X\oplus_p Y)},$  $(z_n^{(1)}), (z_n^{(2)}), \dots,$ $(z_n^{(k)})$ be $(k)-$sequences in $S_{(X\oplus_p Y)}$ with $\|z+\sum_{i=1}^{k}z_n^{(i)}\| \to k+1$ and  $h_1, h_2, \dots, h_{k} \in S_{(X\oplus_pY)^{*}}.$ Clearly $h_j= (f_j, g_j)$  for some $f_j \in B_{X^*}$, $g_j \in B_{Y^*},$ for all $1 \leq j \leq k$, $z=(x,y)$ and  $z_n^{(i)}= (x_n^{(i)}, y_n^{(i)})$ for some $x,x_n^{(i)} \in B_X,$ $y,y_n^{(i)} \in B_Y$ for all $n \in \mathbb{N}$,  $1 \leq i \leq k$. By considering $z_n^{(k+1)}=z$ for all $n \in \mathbb{N}$ and following the similar steps involved in the proof of $(1)$, we obtain $\|x_n^{(i)}\| \to \|x\|$ and $\|y_n^{(i)}\| \to \|y\|$ for all $1 \leq i \leq k$. Now the rest of the proof follows as Case$-(i)$ in the proof of $(1)$.\\
$(3)$: 	Let $X$ be $k_1-$WMLUR, $Y$ be $k_2-$WMLUR and $k= k_1+k_2-1.$ Now, by \Cref{thrm MLUR B_X}, it is enough to show that $B_{X \oplus_p Y}$ is $k-$$w$SCh on $2S_{X \oplus_p Y}.$ Let $z \in S_{X \oplus_p Y},$ $(z_n^{(1)}), (z_n^{(2)}), \dots, (z_n^{(k+1)})$ be $(k+1)-$sequences in $B_{(X\oplus_p Y)}$ such that  $\|2z-z_n^{(i)}\| \to 1$ for all $1 \leq i \leq k+1$ and  $h_1, h_2, \dots, h_{k} \in S_{(X\oplus_pY)^{*}}.$ Clearly $h_j= (f_j, g_j)$ for some $f_j \in B_{X^*}$, $g_j \in B_{Y^*},$ for all $1 \leq j \leq k,$  $z=(x,y)$ and  $z_n^{(i)}= (x_n^{(i)}, y_n^{(i)})$ for some $x, x_n^{(i)} \in B_X$, $y, y_n^{(i)} \in B_Y$ for all $n \in \mathbb{N}$, $1 \leq i \leq k+1.$  Let $1 \leq i \leq k+1$. Note that
\[
	\begin{aligned}
		\Vert 2z-z_n^{(i)}\Vert &= \Vert (\Vert2x- x_n^{(i)}\Vert, \Vert 2y-y_n^{(i)}\Vert) \Vert\\
		& \geq  \Vert (\Vert 2x \Vert- \Vert x_n^{(i)} \Vert, \Vert 2y \Vert- \Vert y_n^{(i)} \Vert ) \Vert\\
		&=  \Vert (\Vert 2x \Vert , \Vert 2y \Vert )- ( \Vert x_n^{(i)}\Vert,  \Vert y_n^{(i)}\Vert) \Vert\\
		&\geq \Vert 2z \Vert - \Vert z_n^{(i)} \Vert\\
		&\geq 1,
	\end{aligned}
\]
which implies $\Vert (\Vert 2x \Vert , \Vert 2y \Vert )- ( \Vert x_n^{(i)}\Vert,  \Vert y_n^{(i)}\Vert) \Vert \to 1.$ Since $B_{(\mathbb{R}^{2}, \|\cdot\|_p)}$ is strongly Chebyshev on $(\mathbb{R}^{2}, \|\cdot\|_p),$ we have $( \Vert x_n^{(i)}\Vert,  \Vert y_n^{(i)}\Vert)  \to (\Vert x \Vert , \Vert y \Vert )$, which further implies  $\Vert x_n^{(i)}\Vert \to  \Vert x \Vert$ and
$\Vert y_n^{(i)} \Vert \to \Vert y \Vert$.
For any subsequence $(n_m)$ of $(n)$, observe that
\[
	\begin{aligned}
		\liminf\limits_{n \to \infty} \left \Vert  2z-z_{n_m}^{(i)} \right\Vert &= 	\liminf\limits_{n \to \infty}  \left \Vert \left(\left\Vert 2x-x_{n_m}^{(i)} \right\Vert, \left\Vert 2y-  y_{n_m}^{(i)} \right\Vert \right)\right\Vert\\
		& \geq \left \Vert \left(\liminf\limits_{n \to \infty} \left\Vert 2x-x_{n_m}^{(i)} \right\Vert,  \liminf\limits_{n \to \infty} \left\Vert 2y - y_{n_m}^{(i)} \right\Vert \right)\right\Vert\\
		&  \geq  \left\Vert \left( \left\Vert  x \right\Vert, \left\Vert y \right\Vert \right) \right \Vert\\
		&= 1,
	\end{aligned}
\]
which implies $\liminf\limits_{n \to \infty} \| 2x-x_{n_m}^{(i)} \| = \|x\|.$ Therefore, $\| 2x-x_n^{(i)} \| \to \|x\|$. If $\|x\|=0,$ by \Cref{remark vol}, we have $D_{k_1}[(x_n^{(\alpha_i)})_{i=1}^{k_1+1}; (f_{\lambda_j})_{j=1}^{k_1}]  \to 0$ for all $\alpha \in \mathcal{S}_{k+1}(k_1+1)$ and $\lambda \in \mathcal{S}_{k}(k_1).$ Assume $\|x\| \neq 0.$ Note that for any $1 \leq i \leq k+1,$ we have
\[
 \left \Vert \frac{x_n^{(i)}}{\|x_n^{(i)}\|}- 2 \frac{x}{\|x\|} \right\Vert = \left\vert \frac{x_n^{(i)}}{\|x_n^{(i)}\|} - \frac{x_n^{(i)}}{\|x\|}+\frac{x_n^{(i)}}{\|x\|} - 2 \frac{x}{\|x\|} \right\vert \leq \left\Vert x_n^{(i)} \right\Vert \left\vert \frac{1}{\|x_n^{(i)}\|}- \frac{1}{\|x\|} \right\vert+ \frac{1}{\|x\|} \left\Vert x_n^{(i)}-2x \right\Vert
\]
and hence $\left\Vert \frac{x_n^{(i)}}{\|x_n^{(i)}\|}- 2 \frac{x}{\|x\|} \right\Vert \to 1.$
Since $X$ is $k_1-$WMLUR, by  \Cref{thrm MLUR B_X}, it follows that $B_X$ is $k_1-$$w$SCh on $2S_X.$ Therefore  $D_{k_1}\left[\left(\frac{x_n^{(\alpha_i)}}{\|x_n^{(\alpha_i)}\|}\right)_{i=1}^{k_1+1}; (f_{\lambda_j})_{j=1}^{k_1}\right]  \to 0,$  further by  \Cref{remark vol} and \Cref{lem samir}, we have $D_{k_1}[(x_n^{(\alpha_i)})_{i=1}^{k_1+1}; (f_{\lambda_j})_{j=1}^{k_1}]  \to 0$  for all $\alpha \in \mathcal{S}_{k+1}(k_1+1)$ and $\lambda \in \mathcal{S}_{k}(k_1).$ Similarly, $D_{k_2}[(y_n^{(\beta_i)})_{i=1}^{k_2+1}; (g_{\mu_j})_{j=1}^{k_2}]  \to 0$ for all $\beta \in \mathcal{S}_{k+1}(k_2+1)$ and $\mu \in \mathcal{S}_{k}(k_2).$ Now by repeating the similar technique involved in Case$-(i)$ of the proof of $(1)$, we obtain $|D_k[(z_n^{(i)})_{i=1}^{k+1}; (h_j)_{j=1}^{k}]| \to 0.$ Hence the proof.
\end{proof}
The next result is an immediate consequence of \Cref{thrm WUR l_p finite rev} and   the fact $(\oplus_pX_i)_{i=1}^{d} \cong (\oplus_pX_i)_{i=1}^{d-1} \oplus_p X_{d}.$
\begin{corollary}\label{coro WUR l_p finite rev}
	Let $d \in \mathbb{Z}^{+}$,  $d>1$ and  $1 < p < \infty.$ Let $X_i$ be a Banach space for all $1 \leq i \leq d$ and $X= (\oplus_pX_i)_{i=1}^{d}.$ If $X_i$ is $k_i-$WUR  (respectively, $k_i-$WLUR, $k_i-$WMLUR) for all $1 \leq i \leq d,$ then  $X$ is $k-$WUR  (respectively, $k-$WLUR, $k-$WMLUR) where $k= \sum_{i=1}^{d}k_i-d+1.$
\end{corollary}

As a consequence of \cite[A.2, A.3, A.4]{Smit1986}, \Cref{coro WUR stability,coro WUR l_p finite for,coro WUR l_p finite rev}, we now present the  necessary and sufficient condition for an infinite $\ell_p-$product space to be $k-$WUR (respectively, $k-$WLUR, $k-$WMLUR).
\begin{theorem}
	Let $1 < p < \infty$, $X_i$ be a Banach space for all $i \in \mathbb{N}$ and $X= (\oplus_pX_i)_{i\in \mathbb{N}}.$ Then the following statements are equivalent.
	\begin{enumerate}
		\item  $X$ is $k-$WUR  (respectively, $k-$WLUR, $k-$WMLUR).
		\item  There exists $j \in \mathbb{N}$ such that  $X_i$ is WUR  (respectively, WLUR, WMLUR) for all $i > j$ and for each  $ i \leq j$ there exists $k_i \in \mathbb{Z}^{+}$ with $\sum_{i=1}^{j}k_i-j+1 \leq k$ such that $X_i$ is $k_i-$WUR  (respectively, $k_i-$WLUR, $k_i-$WMLUR).
	\end{enumerate}
\end{theorem}

We now provide few examples to demonstrate that some of the implications and assertions mentioned in the preceding sections cannot  be reversed in general.

The subsequent example reveals that some implications observed  in \Cref{rem KWUR} and one given immediately below \Cref{def KWLUR} cannot be reversed generally.

\begin{example}\label{eg kWUR}\
	\begin{enumerate}
			\item Consider the space $X= (\ell_2, \| \cdot\|_L)$ from \cite[Example 1]{Smit1978a} and $k \in \mathbb{Z}^{+}.$ In \cite{Smit1978a}, it is proved that $X$ is LUR and reflexive, but not WUR. By \Cref{coro lp WUR}, $\ell_2(X)$ is not $k-$WUR. However, by \cite[Theorem 1.1]{Lova1955}, $\ell_2(X)$ is LUR (hence, $k-$LUR).
		\item Consider the space $X= (\ell_2, \| \cdot\|_A)$ from \cite[Example 3]{Smit1978a} and $k \in \mathbb{Z}^{+}.$ In \cite{Smit1978a}, it is proved that $X$ is strongly rotund (hence, MLUR), but not WLUR. By \Cref{coro lp WUR}, $\ell_2(X)$ is not $k-$WLUR. However, by \cite{Smit1986}, $\ell_2(X)$ is strongly rotund (hence, $k-$strongly rotund and $k-$MLUR).

	\end{enumerate}
\end{example}

We now present an example of a space which is strongly rotund and $(k+1)-$WUR, but not $k-$WLUR  as specified immediately below \Cref{prop k implies k+1}.
\begin{example} \label{eg k+1 WUR}
	For each $x=(x_1, x_2, \dots)$ in $\ell_2,$ define $\|x\|_1= \sup\{\|x\|_{i_1, i_2}: i_1< i_2\},$ where $\|x\|_{i_1, i_2}$ is defined as in \Cref{eg k not k-1 WUR}. Let $(c_n)$ be a decreasing sequence of positive real numbers converges to zero. Define the continuous map $T: (\ell_2, \|\cdot\|_1) \to (\ell_2, \|\cdot\|_2)$ by $T(x_1, x_2, \dots)= (c_2x_2, c_3x_3, \dots).$ Now, define $\|x\|_r^2= \|x\|_1^2+ \|T(x)\|_2^2$ for all $x \in \ell_2.$ Let $B=(\ell_2, \|\cdot\|_r).$ In \cite[Example 2]{NaWa1988}, it is proved that $B$ is $2-$UR and rotund, but not LUR. Now, we will prove that the space $B$ is not WLUR. Let $(e_n)$ be the standard basis of $(\ell_2, \|\cdot\|_2).$ It is easy to see that $\|e_1\|_r=1,$ $\|e_n\|_r \to 1$ and $\|e_1+e_n\|_r \to 2.$ Consider $f=\frac{e_1}{\|e_1\|} \in S_{B^*}.$ Observe that $f(e_1-e_n)= \frac{1}{\|e_1\|}$ for all $n \geq 2.$ Therefore, $e_n-e_1$ does not converge to $0$ weakly. Hence, $B$ is not WLUR. For any $k \in \mathbb{Z}^{+}$ , consider $X=B \oplus_2 B \oplus_2 \dots \oplus_2 B$ $(k)-$times. Clearly, $X$ is strongly rotund. Since $B$ is $2-$WUR, it follows from \Cref{coro WUR l_p finite rev} that $X$ is $(k+1)-$WUR. However, it is easy to see from \Cref{thrm WUR stability} that $X$ is not $k-$WLUR.
\end{example}

The following example illustrates that the implication observed immediately after \Cref{def k-SCh} cannot be reversed in general. The example also shows that the converse of \Cref{prop kWLUR prox convex} not necessarily true.

\begin{example}\label{eg converse Chev2} Let $k \in \mathbb{Z}^{+}$ and  $X$ be a $k-$strongly rotund, but not $k-$WLUR space (see, \Cref{eg kWUR,eg k+1 WUR}). Therefore, by \cite[Theorem 2.10]{VeRS2021}, every closed convex subset of $X$ is $k-$SCh on $X,$ in particular $B_X$ is $k-$$w$SCh on $2S_X.$ However, by \Cref{thrm k-WUR B_X}, $B_X$ is not $k-$$w$USCh on $2S_X.$
\end{example}

As mentioned after \Cref{prop WUR Rotund}, the following example demonstrate that  $k-$weakly uniform rotundity of $X$ does not imply that the space $X^{**}$ is $k-$WMLUR.

\begin{example}\label{eg WUR rotund suff}
	Let $Z=(c_0, \|\cdot\|_{\infty})$and $k \in \mathbb{Z}^{+}.$ By \cite[Chapter II, Corollary 6.9]{DeGZ1993}, $Z$ admits an equivalent norm (say, $\|\cdot\|_r)$ such that $Y= (c_0, \|\cdot\|_r)$ is WUR. Since it is proved in \cite{AlDi1985} that $\ell_{\infty}$ does not have any equivalent WMLUR renorming, we have $Y^{**}$ is not WMLUR. Consider the Banach space $X= l_2(Y).$ Then, by \cite[A.2]{Smit1986}, $X$ is WUR (hence, $k-$WUR). Clearly, $X^{**} \cong l_2(Y^{**}).$ Therefore, by \Cref{coro lp WUR}, $X^{**}$ is not $k-$WMLUR.
	\end{example}


\begin{thebibliography}{99}

\bibitem{AlDi1985}
G.~A. Aleksandrov and I.~P. Dimitrov.
\newblock On the equivalent weakly midpoint locally uniformly rotund renorming
  of the space {$l_\infty$}.
\newblock In {\em Proc. 14th Spring Conference of the Union of Bulgarian
  Mathematicians, Sunny Beach}, pages 189--191 (in Russian). 1985.

\bibitem{BLLN2008}
P.~Bandyopadhyay, Y.~Li, B.~L. Lin, and D.~Narayana.
\newblock Proximinality in {B}anach spaces.
\newblock {\em J. Math. Anal. Appl.}, 341(1):309--317, 2008.

\bibitem{Clar1936}
J.~A. Clarkson.
\newblock Uniformly convex spaces.
\newblock {\em Trans. Amer. Math. Soc.}, 40(3):396--414, 1936.

\bibitem{DeGZ1993}
R.~Deville, G.~Godefroy, and V.~Zizler.
\newblock {\em Smoothness and renormings in {B}anach spaces}.
\newblock Pitman Monographs and Surveys in Pure and Applied Mathematics, vol.
  64. Longman Scientific \& Technical, Harlow, 1993.

\bibitem{DuSh2018}
S.~Dutta and P.~Shunmugaraj.
\newblock Weakly compactly {LUR} {B}anach spaces.
\newblock {\em J. Math. Anal. Appl.}, 458(2):1203--1213, 2018.

\bibitem{GaTh2022}
P.~Gayathri and V.~Thota.
\newblock Characterizations of weakly uniformly rotund {B}anach spaces.
\newblock {\em J. Math. Anal. Appl.}, 514(1):Paper No. 126298, 15, 2022.

\bibitem{GaTh2023}
P.~Gayathri and V.~Thota.
\newblock On geometric and best approximation properties of {$k$}-{UR} and its generalizations.
\newblock {\em Banach J. Math. Anal.}, 17(2):Paper No. 29, 32, 2023.

\bibitem{GeSu1981}
R.~Geremia and F.~Sullivan.
\newblock Multidimensional volumes and moduli of convexity in {B}anach spaces.
\newblock {\em Ann. Mat. Pura Appl.}, 127:231--251, 1981.

\bibitem{He1997}
R.~Y. He.
\newblock {$K$}-strongly convex and locally {$K$}-uniformly smooth spaces.
\newblock {\em J. Math. (Wuhan)}, 17(2):251--256, 1997.

\bibitem{HoJo2013}
R.~A. Horn and C.~R. Johnson.
\newblock {\em Matrix analysis}.
\newblock Cambridge University Press, Cambridge, second edition, 2013.

\bibitem{KaVe2018}
S.~Kar and P.~Veeramani.
\newblock On {$k$}-uniformly rotund spaces and spaces with property {$k$}-{UC}.
\newblock {\em J. Nonlinear Convex Anal.}, 19(7):1263--1273, 2018.

\bibitem{LiYu1985}
B.~L. Lin and X.~T. Yu.
\newblock On the {$k$}-uniform rotund and the fully convex {B}anach spaces.
\newblock {\em J. Math. Anal. Appl.}, 110(2):407--410, 1985.

\bibitem{LiZZ2018}
C.~Liu, Z.~Zhang, and Y.~Zhou.
\newblock A note in approximative compactness and midpoint locally
  {$k$}-uniform rotundity in {B}anach spaces.
\newblock {\em Acta Math. Sci. Ser. B (Engl. Ed.)}, 38(2):643--650, 2018.

\bibitem{Lova1955}
A.~R. Lovaglia.
\newblock Locally uniformly convex {B}anach spaces.
\newblock {\em Trans. Amer. Math. Soc.}, 78(1):225--238, 1955.

\bibitem{Megg1998}
R.~E. Megginson.
\newblock {\em An introduction to {B}anach space theory}.
\newblock Graduate Texts in Mathematics, vol. 183. Springer-Verlag, New York,
  1998.

\bibitem{NaWa1988}
C.~X. Nan and J.~H. Wang.
\newblock On the {L{$k$}}-{UR} and {L}-{$k$}{R} spaces.
\newblock {\em Math. Proc. Cambridge Philos. Soc.}, 104(3):521--526, 1988.

\bibitem{Read2018}
C.~J. Read.
\newblock Banach spaces with no proximinal subspaces of codimension 2.
\newblock {\em Israel J. Math.}, 223(1):493--504, 2018.

\bibitem{Rmou2017}
M.~Rmoutil.
\newblock Norm-attaining functionals need not contain 2-dimensional subspaces.
\newblock {\em J. Funct. Anal.}, 272(3):918--928, 2017.

\bibitem{Sing1960}
I.~Singer.
\newblock On the set of the best approximations of an element in a normed
  linear space.
\newblock {\em Rev. Math. Pures Appl.}, 5:383--402, 1960.

\bibitem{Smit1978a}
M.~A. Smith.
\newblock Some examples concerning rotundity in {B}anach spaces.
\newblock {\em Math. Ann.}, 233(2):155--161, 1978.

\bibitem{Smit1986}
M.~A. Smith.
\newblock Rotundity and extremity in {$l^p(X_i)$} and {$L^p(\mu,X)$}.
\newblock In {\em Geometry of normed linear spaces, {U}rbana-{C}hampaign,
  {I}ll, 1983, In Contemp. Math.}, vol. 52, pages 143--162. Amer. Math. Soc.,
  Providence, RI, 1986.

\bibitem{SmTu1990}
M.~A. Smith and B.~Turett.
\newblock Some examples concerning normal and uniform normal structure in
  {B}anach spaces.
\newblock {\em J. Austral. Math. Soc. Ser. A}, 48(2):223--234, 1990.

\bibitem{Sull1979}
F.~Sullivan.
\newblock A generalization of uniformly rotund {B}anach spaces.
\newblock {\em Canadian J. Math.}, 31(3):628--636, 1979.

\bibitem{Suya2000}
Suyalatu.
\newblock On some generalization of local uniform smoothness and dual concepts.
\newblock {\em Demonstratio Math.}, 33(1):101--108, 2000.

\bibitem{Veen2021}
M.~Veena~Sangeetha.
\newblock Geometry of product spaces.
\newblock {\em J. Math. Anal. Appl.}, 503(1):Paper No. 125285, 23, 2021.

\bibitem{VeRS2021}
M.~Veena~Sangeetha, M.~Radhakrishnan, and S.~Kar.
\newblock On $k$-strong convexity in {B}anach spaces.
\newblock {\em J. Convex Anal.}, 28(4):1193--1210, 2021.

\bibitem{VeVe2018}
M.~Veena~Sangeetha and P.~Veeramani.
\newblock Uniform rotundity with respect to finite-dimensional subspaces.
\newblock {\em J. Convex Anal.}, 25(4):1223--1252, 2018.

\bibitem{Smul1939}
V.~L.~\v{S}mulian.
\newblock On the principle of inclusion in the space of the type {$({\rm B})$}.
\newblock {\em Rec. Math. [Mat. Sbornik] N.S.}, 5(47):317--328, 1939.

\bibitem{XiLi2004}
J.~Xian and Y.~J. Li.
\newblock {$K$}-very-convex spaces and {$K$}-very-smooth spaces.
\newblock {\em J. Math. Res. Exposition}, 24(3):483--492, 2004.

\bibitem{YuWa1990}
X.~T. Yu and J.~P. Wang.
\newblock On moduli of {$k$}-rotundity and {$k$}-convexity.
\newblock {\em Chinese Ann. Math. Ser. A}, 11(2):212--222, 1990.

\bibitem{ZhLZ2015}
Z.~Zhang, C.~Liu, and Y.~Zhou.
\newblock Some examples concerning proximinality in {B}anach spaces.
\newblock {\em J. Approx. Theory}, 200:136--143, 2015.

\end{thebibliography}
\end{document}